\documentclass{amsart}
\usepackage{amssymb}
\usepackage{amsmath,amscd}

\newtheorem{theorem}{Theorem}[section]
\newtheorem{lemma}[theorem]{Lemma}
\newtheorem{proposition}[theorem]{Proposition}
\newtheorem{corollary}[theorem]{Corollary}
\newtheorem*{theorem*}{Theorem}

\theoremstyle{definition}
\newtheorem{question}[theorem]{Question}
\newtheorem{definition}[theorem]{Definition}
\newtheorem{example}[theorem]{Example}

\newtheorem{choice}[theorem]{Choice}
\newtheorem{assumption}[theorem]{Assumption}

\theoremstyle{remark}
\newtheorem{remark}[theorem]{Remark}

\newcommand{\euO}{\mathfrak O}
\newcommand{\euP}{\mathfrak P}
\newcommand{\euD}{\mathfrak D}
\newcommand{\euA}{\mathfrak A}

\newcommand{\euH}{\mathfrak H}

\newcommand{\euc}{\mathfrak c}
\newcommand{\eub}{\mathfrak b}

\newcommand{\eua}{\mathfrak a}

\newcommand{\Z}{\mathbb{Z}}

\newcommand{\bQ}{\mathbb Q}
\newcommand{\bZ}{\mathbb Z}
\newcommand{\bF}{\mathbb F}

\newcommand{\Spn}{\mathbb{S}_{p^n}}

\newcommand{\Gal}{\mathrm{Gal}}

\newcommand{\Tr}{\mathrm{Tr}}
\newcommand{\chr}{\mathrm{char}}

\begin{document}
\title[Galois Scaffold]{Sufficient conditions for large Galois scaffolds}

\author{Nigel P.~Byott}
\address{Department of Mathematics, University of Exeter, Exeter 
EX4 4QF U.K.}  \email{N.P.Byott@exeter.ac.uk}

\author{G.~Griffith Elder}
\address{Department of Mathematics, University of Nebraska at Omaha, Omaha NE 68182-0243 U.S.A.}  
\email{elder@unomaha.edu} 

\subjclass[2010]{Primary 11S15, Secondary 11R33, 16T05}
\keywords{Galois module structure, Hopf order}
\date{\today}

\bibliographystyle{amsalpha}

\begin{abstract}
Let $L/K$ be a finite, Galois, totally ramified $p$-extension of
complete local fields with perfect residue fields of characteristic
$p>0$.  In this paper, we give conditions, valid for any Galois
$p$-group $G=\mbox{Gal}(L/K)$ (abelian or not) and for $K$ of either
possible characteristic ($0$ or $p$), that are sufficient for the
existence of a Galois scaffold.  The existence of a Galois scaffold
makes it possible to address questions of integral Galois module
structure, which is done in a separate paper \cite{byott:A-scaffold}.
But since our conditions can be difficult to check, we specialize to
elementary abelian extensions and extend the main result of
\cite{elder:scaffold} from characteristic $p$ to characteristic $0$.
This result is then 
applied, using a result of Bondarko, to the construction
of new Hopf orders over the valuation ring $\euO_K$ that lie in $K[G]$ for
$G$ an elementary abelian $p$-group.
\end{abstract}

\maketitle

\section{Introduction}
Let $p$ be prime, $\kappa$ be a perfect field of characteristic $p$,
and $K$ be a local field with residue field $\kappa$. Let $L$ be a
totally ramified Galois extension of $K$ with $G=\mbox{Gal}(L/K)$ of
degree $p^n$ for some $n>0$, and let $\euO_L$ be the ring of integers
of $L$ ({\em i.e.} its valuation ring). {\em Local integral Galois module
  theory} asks a question that is a consequence of three classical
results: the Normal Basis Theorem, which states that $L$ {\em is free}
over the group algebra $K[G]$; a result of E.~Noether \cite{noether},
which concludes that, because $L/K$ is wildly ramified, $\euO_L$ {\em
  is not free} over the group ring $\euO_K[G]$; and a local version of a
result of H.~W.~Leopoldt \cite{leopoldt}, which states that for
absolute abelian extensions of the $p$-adic numbers ({\em i.e.}
$K=\mathbb{Q}_p$), $\euO_L$ {\em is free} over its associated order
$$\euA_{L/K}=\{\alpha\in K[G]:\alpha\euO_L\subseteq \euO_L\},$$ the
largest $\euO_K$-order in the group algebra $K[G]$ for which $\euO_L$
is a module.  
\begin{question}\label{ques-GMT}
When is the ring of integers $\euO_L$ free over its associated
order $\euA_{L/K}$?
\end{question}

Restrict for the moment to the situation where $K$ is a finite
extension of $\mathbb{Q}_p$.  The earliest answers here showed us that
unless $K=\mathbb{Q}_p$, $\euO_L$ need not be free over $\euA_{L/K}$,
which is why the question is currently asked in this
way. Additionally, those early answers suggested a form that we might
expect the answers to take. Based upon work of F.~Bertrandias and
M.-J.~Ferton \cite{ferton} when $L/K$ is a $C_p$-extension, and
B.~Martel \cite{martel} when $L/K$ is a $C_2\times C_2$-extension, we
might expect the answer to Question \ref{ques-GMT}, necessary and
sufficient conditions for $\euO_L$ to be free over $\euA_{L/K}$, to be
expressed in terms of the ramification numbers associated with the
extension (integers $i$ such that $G_i\neq G_{i+1}$ where $G_i$ is the
$i$th ramification group \cite[IV \S1]{serre:local}). There have not
been that many further results in this direction. Still,
\begin{enumerate}
\item When $L/K$ is an abelian extension, and the ring of integers is
  replaced with the inverse different $\euD_{L/K}^{-1}$, \cite[Theorem
    3.10]{byott:JNTB} determines necessary conditions, in terms of
  ramification numbers, for the inverse different to be free over its
  associated order.
\item When $K/\bQ_p$ is unramified and $L/K$ is a totally ramified
  abelian extension (not necessarily of $p$-power degree), D.~Burns
  \cite{burns:AIF} investigated freeness of ideals in $\euO_L$ over
  their associated orders in $K[G]$. This was extended in
  \cite{burns:CMH} to the case where $K/\bQ_p$ can be ramified, but
  associated orders are considered in $\bQ_p[G]$ (or, more generally,
  in $E[G]$, where $E \subseteq K$ and $E/\bQ_p$ is unramified). In
  both these situations, the existence of {\em any} ideal free over
  its associated order forces strong restrictions on the ramification
  of the extension $L/K$.
\item When $L/K$ is a special type of cyclic Kummer extension, namely
  $L=K(\sqrt[p^n]{1+\beta})$ for some $\beta\in K$ with $p\nmid
  v_K(\beta)>0$, where $v_K$ is the normalized valuation on $K$, Y.~Miyata determines necessary and sufficient
  conditions for $\euO_L$ to be free over $\euA_{L/K}$ in terms of
  $v_K(\beta)$. These conditions can be restated
  in terms of ramification numbers  \cite{miyata:cyclic2}.
\item Finally, we move into characteristic $p$ with
  $K=\kappa((t))$. When $L/K$ is a special type of elementary abelian
  extension, namely {\em near one-dimensional}, and thus has a {\em Galois
    scaffold} \cite{elder:scaffold}, necessary and sufficient
  conditions for $\euO_L$ to be free over $\euA_{L/K}$ are given
  in terms of ramification numbers \cite{byott:scaffold}.
\end{enumerate}
Interestingly, the conditions on the ramification numbers in
\cite{byott:scaffold} agree with those given in
\cite{miyata:cyclic2} (as
translated by \cite{byott:QJM}).

The purpose of this paper is to extend the setting where Galois
scaffolds have been proven to exist, namely
\cite{elder:scaffold,elder:sharp-crit}: from characteristic $p$ to
characteristic $0$, and from elementary abelian (or cyclic of degree
$p^2$) $p$-groups to all $p$-groups (abelian or not).
We do this, in Theorem \ref{main}, by determining conditions
sufficient for a Galois scaffold to exist that are independent of
characteristic and of Galois group.  When an extension $L/K$ satisfies
the hypotheses of Theorem \ref{main} and thus possesses a Galois
scaffold, the answer to Question \ref{ques-GMT} is provided in
\cite{byott:A-scaffold}, where necessary and sufficient conditions are
given, not just for $\euO_L$, but for each fractional ideal $\euP_L^i$
of $\euO_L$, to be free over its associated
order. Indeed, stronger questions, such as those asked by B.~de Smit
and L.~Thomas in \cite{desmit:2}, are also addressed.  Each answer is
given in terms of ramification numbers.

On the other hand, given only the generators
of an extension,  it is not easy to determine whether the extension
satisfies the conditions of 
Theorem \ref{main}. Thus in \S3, we describe, in terms of 
Artin-Schreier generators, arbitrarily large elementary abelian
$p$-extensions that do satisfy the conditions of
Theorem \ref{main} and thus possess a Galois scaffold.  
In characteristic $0$, the result is new. These are the
analogs of the near one-dimensional elementary abelian extensions of
\cite{elder:scaffold}. In \S4, to illustrate the level of explicit detail that
is then possible when the results of this paper are combined with
\cite{byott:A-scaffold}, we include results in characteristic $0$, on
the structure of 
$\euP_L^i$ over its associated order, for certain families of
elementary abelian extensions that are of common interest.

Finally, to illustrate the utility of our results beyond local
integral Galois module theory, we explain how the results of this
paper combined with \cite{bondarko,byott:A-scaffold} can be used to
attack the difficult problem of classifying Hopf orders in the group
algebra $K[G]$ for $G$ some $p$-group. This is an old problem. The
first result in this direction is that of Tate and Oort
\cite{tate:oort} for Hopf orders of rank $p$. And yet, the
classifications for $G\cong 
C_p^3,C_{p^3}$ remain incomplete \cite[Proposition
  15]{childs:elem-abel}, \cite[Theorem 5.4]{underwood:childs}.
Notably, the Hopf orders that are missing for $G\cong C_p^3$ include
those which are realizable as the associated orders of valuation
rings, and it is precisely such Hopf orders that the
results of this paper are designed to produce.  Indeed, \S5 can be
viewed as providing a model, given any $p$-group $G$, for the
construction of such ``realizable'' Hopf orders in $K[G]$. As such, it
provides motivation for future work identifying extensions
that satisfy the hypotheses of Theorem \ref{main}.

We close this introduction by pointing out that our work is somewhat
similar in spirit to that of Bondarko \cite{bondarko,bondarko:ideals,
  bondarko:leo}, who also considers the existence of ideals free over
their associated orders in the context of totally ramified Galois
extensions of $p$-power degree. Bondarko introduces the class of
semistable extensions. Any such extension contains at least one ideal
free over its associated order, and all such ideals can be determined
from numerical data. Moreover, any abelian extension containing an
ideal free over its associated order, and satisfying certain
additional assumptions, must be semistable. Abelian semistable
extensions can be completely characterized in terms of the Kummer
theory of (one-dimensional) formal groups. The precise relationship
between Bondarko's results and our own remains to be explored.

\subsection{Discussion of our approach} \label{background}
The existence of a Galois scaffold addresses an issue, which
is illustrated in the following two
examples. Let $v_K,v_L$ denote the normalized valuations for $K,L$,
respectively. Choose $\pi\in K$ with $v_K(\pi)=1$.

\begin{example} \label{ex1}
Fix a local field $K$ and suppose that $L/K$ is
a totally ramified Galois extension of degree $p$. Let $\sigma$ generate $G$.
Then $L/K$ has a unique ramification
break $b$, and this is characterized by the property that, for all $\alpha \in
L\backslash\{0\}$,
$$ v_L( (\sigma-1) \cdot \alpha) \geq v_L(\alpha) + b, \mbox{ with
  equality if } p \nmid v_L(\alpha). $$ 
Let us suppose for simplicity that $b \equiv -1 \bmod p$, say $b=pr-1$
with $r \geq 1$. Fix a uniformizing parameter $\pi$ of $K$, and let
$\Psi=(\sigma-1)/\pi^r \in K[G]$. Pick any $\rho \in L$ with
$v_L(\rho)=p-1$. Then, for $0 \leq j \leq p-1$, we have $v_L(\Psi^j
\cdot \rho)=p-1-j$. Thus $\Psi$ typically reduces valuations by $1$,
and the $\Psi^j \cdot \rho$ for $0 \leq j \leq p-1$ form an
$\euO_K$-basis of $\euO_L$. Two conclusions follow: firstly, that the
$\Psi^j$ form an $\euO_K$-basis of the associated order $\euA_{L/K}$,
and, secondly, that $\euO_L$ is a free module over $\euA_{L/K}$,
generated by any element $\rho$ of valuation $p-1$. 
\end{example}

Example \ref{ex1} in itself is nothing
new. Indeed, far more comprehensive treatments of the valuation ring
of an extension of degree $p$ are given in \cite{ferton, bbferton} for
the characteristic $0$ case, and in \cite{aiba,desmit:2} for
characteristic $p$. (See also \cite{ferton:ideals} for arbitrary
ideals in characteristic $0$, and \cite{huynh} and \cite{Maria} for
the corresponding 
problem in characteristic $p$.) We now consider what happens if we try
to make the same argument for a larger extension.

\begin{example} \label{ex2}
Let $L/K$ be a totally ramified extension of degree $p^2$. We now
have two ramification breaks $b_1 \leq b_2$ (in the lower numbering),
and we necessarily have $b_1 \equiv b_2 \bmod p$. For simplicity we assume that $b_i \equiv -1 \bmod p^2$, say 
$b_i=r_ip^2-1$, for $i=1$, $2$. We can
then find elements $\sigma_1$, $\sigma_2$ which generate $\Gal(L/K)$
and for which, setting $\Psi'_i=(\sigma_i-1)/\pi^{r_i}$, we have
$$ v_L(\Psi'_i \cdot \alpha) \geq v_L(\alpha)-1 \mbox{ for } i=1, 2,
\mbox{ with equality if }
 p \nmid v_L(\alpha) $$
whenever $\alpha \in L\backslash\{0\}$. Thus $\Psi'_1$ and $\Psi'_2$ both typically reduce valuations by
$1$, but this does not enable us to determine
$\euA_{L/K}$. Now suppose that we could
replace $\Psi'_1$,  $\Psi'_2$ with elements $\Psi_1$,  $\Psi_2$ such
that, for some suitable choice of $\rho \in L$ with $v_L(\rho)=p^2-1$, we
had
\begin{equation} \label{int-ex}
   v_L(\Psi_1^{j_1} \Psi_2^{j_2} \cdot \rho)= p^2-1 -j_1 p -j_2 \mbox{ for
  all } 0 \leq j_1, j_2, \leq p-1. 
\end{equation}
Thus, at least on the family of elements of $L$ of the form
$\Psi_i^{i_1} \Psi_2^{i_2} \cdot \rho$, we can say that $\Psi_1$
typically reduces valuations by $p$, whilst $\Psi_2$ typically reduces
valuations by $1$.  We could then deduce that the elements
$\Psi_1^{j_1} \Psi_2^{j_2}$ form an $\euO_K$-basis of $\euA_{L/K}$,
and that $\euO_L$ is free over $\euA_{L/K}$ on the generator
$\rho$. Such elements $\Psi_i$  would essentially
constitute a Galois scaffold. 
\end{example}

The reason that we cannot determine $\euA_{L/K}$ in Example \ref{ex2}
using the original elements $\Psi'_1$ and $\Psi'_2$ is that we have
insufficient information about their effect on elements of $L$ whose
valuation is divisible by $p$ but not by $p^2$. It is because of this
problem that early attempts to treat other cases in the same manner as
degree $p$ extensions achieved only limited success. (See for instance
\cite{ferton:cyclic} for cyclic extensions of degree $p^n$, $n \geq
2$, and, temporarily relaxing the condition that $L/K$ has $p$-power
degree, \cite{ferton:dihedral, berge:dihedral} for dihedral extensions
of degree $2p$. A complete treatment of biquadratic extensions of
$2$-adic fields was, however, given in \cite{martel}.)

\subsection{Intuition of a scaffold}\label{intuition}
The intuition underlying a scaffold can be explained, as is done in
\cite{byott:A-scaffold}, somewhat informally.  For the convenience of
the reader, we replicate it here: Given any positive integers $b_i$
for $1\leq i\leq n$ such that $p\nmid b_i$ (think of lower
ramification numbers), there are elements $X_i\in L$ such that
$v_L(X_i)=-p^{n-i}b_i$.  Since the valuations, $v_L$, of the monomials
$$\mathbb{X}^a=X_n^{a_{(0)}}X_{n-1}^{a_{(1)}}\cdots X_1^{a_{(n-1)}}:
0\leq a_{(i)}<p,$$ provide a complete set of residues modulo $p^n$ and
$L/K$ is totally ramified of degree $p^n$, these monomials provide a
convenient $K$-basis for $L$.  The action of the group ring $K[G]$ on $L$ is clearly
determined by its action on the monomials $\mathbb{X}^a$.   So if there were
$\Psi_i\in K[G]$ for $1\leq i\leq n$ such that each $\Psi_i$
acts on the monomial basis element $\mathbb{X}^a$ of $L$ as if
it were the differential operator $d/dX_i$ and the $X_i$ were
independent variables, namely
\begin{equation} \label{deriv-exact}
  \Psi_i\mathbb{X}^a=a_{(n-i)}\mathbb{X}^a/X_i, 
\end{equation}
then the monomials in the $\Psi_i$ (with exponents bound $< p$) would
furnish a convenient basis for $K[G]$ whose effect on the
$\mathbb{X}^a$ would be easy to determine. As a consequence, the
determination of $\euA_{L/K}$, and of the structure of $\euO_L$ over
$\euA_{L/K}$ would be reduced to a purely numerical calculation
involving the $b_i$. This remains true if \eqref{deriv-exact} is
loosened to the congruence,
\begin{equation}  \label{deriv-cong} 
   \Psi_i\mathbb{X}^a\equiv a_{(n-i)}\mathbb{X}^a/X_i\bmod
   (\mathbb{X}^a/X_i) \euP_L^{\euc} 
\end{equation}
for a sufficiently large ``precision'' $\euc$.  The $\Psi_i$, together
with the $\mathbb{X}^a$, constitute a Galois scaffold on $L$.

The formal definition of a scaffold \cite[Definition
  2.3]{byott:A-scaffold} generalizes this situation.  Indeed, given
this intuitive connection with differentiation, it is perhaps not
surprising that scaffolds can be constructed from higher derivations
on an inseparable extension, as is done in
\cite[\S5]{byott:A-scaffold}. Ironically, with this perspective it may
now be surprising that they can be constructed for Galois extensions
under the action of $K[G]$. Yet, this is where they were first
constructed \cite{elder:scaffold}.

\section{Main Result: Construction of Galois scaffold} \label{suf-sect} 

Recall that $K$ is a complete local field whose residue field is
perfect of characteristic $p>0$, and that $L/K$ is a totally ramified
Galois extension of degree $p^n$.  Relabel now, so that $L/K=K_n/K_0$.
Following common practice, we use subscripts to denote field of
reference. So $v_n\colon K_n \twoheadrightarrow \mathbb{Z}
\cup\{\infty\}$ is the normalized valuation, and $\pi_n$ is a prime
element of $K_n$ with $v_n(\pi_n)=1$. The valuation ring of $K_n$ is
denoted by $\euO_n$ with maximal ideal $\euP_n$.  Let $G_i=\{\sigma\in
G:v_n((\sigma-1)\pi_n)\geq i+1\}$ be the $i$th group in the
ramification filtration of the Galois group $G=\mbox{Gal}(K_n/K_0)$.

In this section we construct a Galois scaffold, in Theorem \ref{main},
for extensions $K_n/K_0$ that satisfy three assumptions, which in turn
depend upon two choices.  For emphasis, we repeat here that $K_0$ may
have characteristic $0$ or $p$. The Galois group $G$ can be
nonabelian, as well as abelian.  We also point out that, except
for Assumption \ref{ass3}, all these choices and assumptions appear in
\cite{elder:scaffold}. Our first choice organizes the extension.
\begin{choice}\label{choice1}
Choose a composition series for $G$ that refines the ramification
filtration: $\{H_i\}\supseteq \{G_i\}$ such that $H_0=G$,
$H_n=\{e\}$ and $H_{i-1}/H_i\cong C_p$. Furthermore, choose one
element to represent each degree $p$ quotient: $\sigma_i\in
H_{i-1}\setminus H_i$.
\end{choice}
Let $K_i=K_n^{H_i}$ be the fixed field of $H_i$, and let
$b_i=v_n((\sigma_{i}-1)\pi_n)-1$.  Because of Choice \ref{choice1}, we
can see, using \cite[IV\S1]{serre:local}, that the multiset
$B=\{b_i:1\leq i\leq n\}$ is the set of lower ramification numbers,
namely the set of subscripts $i$ with $G_i\supsetneq G_{i+1}$, with
multiplicity $\log_p|G_{b_i}/G_{b_i+1}|$.  In particular, $b_1\leq
b_2\leq \cdots \leq b_n$, $\{b_i:j<i\leq n\}$ is the ramification
multiset for $K_n/K_j$, $\{b_i:0<i\leq j\}$ is the ramification
multiset for $K_j/K_0$, and $b_j$ is the lower ramification number for
$K_j/K_{j-1}$.  The set of upper ramification numbers $\{u_i\}$ is
determined by
\begin{equation}\label{upper}
u_i=b_1+\frac{b_2-b_1}{p}+\cdots +\frac{b_i-b_{i-1}}{p^{i-1}}
=(p-1)\left (\frac{b_1}{p}+\cdots +\frac{b_{i}}{p^{i}}\right
)+\frac{b_i}{p^{i}}\end{equation} 
\cite[IV\S3]{serre:local}.
Furthermore note that $\{u_i:0<i\leq j\}$ is the set of
upper ramification numbers for $K_j/K_0$, but that $\{u_i:j<i\leq n\}$
is not necessarily the set of upper ramification numbers for $K_n/K_j$.

Our first assumption is weak, as it does not eliminate any extension
in characteristic $p$. In characteristic $0$, it eliminates only maximally ramified extensions, {\em i.e.} those
cyclic extensions $K_n/K_0$ where $K_0$ contains the $p$th roots of
unity and $K_1=K_0(\sqrt[p]{\pi_0})$ for some prime element $\pi_0\in
K_0$ \cite[IV\S2 Exercise 3]{serre:local}.
\begin{assumption}\label{ass1}
$p\nmid b_1$.
\end{assumption}

Now we choose generators for $K_n/K_0$ based upon Choice
\ref{choice1}.  Since the valuation $v_j$ is normalized so that
$v_j(K_j^\times)=\mathbb{Z}$, there are $Y_j\in K_j$ with
$v_j(Y_j)=-b_j$. Since $v_j((\sigma_j-1)Y_j)=0$, a unit $u\in K_0$
exists such that $v_j( \sigma_j-1)u^{-1}Y_j - 1)>0$.

\begin{choice}\label{choice2} For each $1\leq j\leq n$, choose
$X_j\in K_j$ such that $v_j(X_j)=-b_j$ and
$v_j((\sigma_j-1)X_j-1)>0$.
\end{choice}
Since $b_j\equiv b_1\bmod p$ \cite[IV\S2]{serre:local}, we have $p\nmid
b_j$ and therefore $K_j=K_{j-1}(X_j)$.  
\begin{remark}
Since $p\nmid b_j$, we could choose $X_j$ so that, additionally, it
satisfies an Artin-Schreier equation $X_j^p-X_j\in K_{j-1}$ \cite[III
  \S2 Proposition 2.4]{fesenko}. In characteristic $0$, this is a
result of MacKenzie and Whaples. We do not make this a requirement
however, since we do not need to use this fact.
\end{remark}

Define the
binomial coefficient
$$\binom{Y}{i}=\frac{Y(Y-1)\cdots(Y-i+1)}{i!}\in \mathbb{Q}[Y]$$ for $i\geq 0$,
and $\binom{Y}{i}=0$ for $i<0$.
For integers $-p<v_{(i)}<p$ form the $n$-tuple,
$\vec{v}=(v_{(n-1)},\ldots ,v_{(0)})$. Define
$\rho_{\vec{v}}=\prod_{i=1}^n\binom{X_i}{v_{(n-i)}}\in K_n$.
Thus $\rho_{\vec{v}}=0$, if there is an $0\leq i<n$ with $v_{(i)}<0$.
Define the partial order $\preceq$ on 
$n$-tuples: Given $\vec{v}$, $\vec{w}$,
$$\vec{v}\preceq \vec{w}\iff v_{(i)}\leq w_{(i)}\mbox{ for all } 0\leq
i<n.$$ Thus $\rho_{\vec{v}}\neq 0$ if and only if $\vec{0}\preceq
\vec{v}$.  Now restrict to vectors $(a_{(n-1)},\ldots ,a_{(0)})$ of
the base-$p$ coefficients of integers $0\leq a<p^n$, and identify each
$a= \sum_{i=1}^na_{(n-i)}p^{n-i}$ where $0\leq a_{(s)}<p$ with its corresponding
vector.  (It is convenient to index the base-$p$ digits as
$a_{(n-i)}$, where increasing values of $i$ correspond to decreasing
powers of $p$.)  Define
$$\rho_{a}=\prod_{i=1}^n\binom{X_i}{a_{(n-i)}}\in K_n.$$ 
Furthermore, define
$$ \eub(a) : = - v_n(\rho_{a}) = -\sum_{i=1}^na_{(n-i)}p^{n-i}b_i. $$
Because the $b_i$ are relatively
prime to $p$, $\{-\eub(a):0\leq a<p^n\}$ is a complete set of residues
modulo $p^n$. As a result, $\{\rho_{a}:0\leq a<p^n\}$ is a $K_0$-basis
for $K_n$.  Since $-\eub$ maps the residues modulo $p^n$ onto the
residues modulo $p^n$, it has an inverse $\eua$: For each $t\in
\mathbb{Z}$, we define $\eua(t)$ to be the unique integer satisfying
$$ 0\leq \eua(t)<p^n, \qquad t=-\eub(\eua(t))+p^nf_t 
   \mbox{ for some } f_t\in \mathbb{Z}. $$  
Note that $\eua(0)=0$. Using this notation, we normalize our
$K_0$-basis for $K_n$ as follows.
\begin{definition}\label{lambda}
Let $\lambda_t=\pi_0^{f_t}\rho_{\eua(t)}$, where $\pi_0$ is a fixed
prime element in $K_0$. Thus $v_n(\lambda_t)=t$ for all $t$,
$\lambda_{t+p^n}=\pi_0\lambda_t$, and $\{\lambda_t:0\leq t<p^n\}$ is
an $\euO_0$-basis for $\euO_n$.
\end{definition}

We need to discuss Galois action. Choice
\ref{choice2} means that
$v_j((\sigma_i-1)X_j)=b_i-b_j$ for $1\leq i\leq j\leq n$.
Recall that $b_i- b_j\equiv 0\bmod p$ \cite[IV\S2]{serre:local}.
Since $K_j/K_{j-1}$ is ramified, there are elements $\mu_{i,j}\in K_{j-1}$
and $\epsilon_{i,j}\in K_j$
such that 
\begin{equation}\label{sigmaX}
(\sigma_i-1)X_j=\mu_{i,j}+\epsilon_{i,j}
\end{equation}
with $v_j(\mu_{i,j})=b_i-b_j<v_j(\epsilon_{i,j})$.  We consider
$\mu_{i,j}$ to be the {\em main term}, with $\epsilon_{i,j}$ the {\em
  error term}. Observe that
$v_j([(\sigma_i-1)-\mu_{i,j}(\sigma_j-1)]X_j)>v_j((\sigma_i-1)X_j)$.
We would like 
this observation to be a statement about an element
$\mu_{i,j}(\sigma_j-1)\in K_0[G]$
that approximates the effect of $(\sigma_i-1)$. So
observe that
if $p^j\mid v_j(\mu_{i,j})=b_i-b_j$, then we may choose 
$\mu_{i,j}\in
K_0$. The condition
$p^j\mid b_i-b_j$ for all $1\leq i\leq j\leq
n$
is equivalent to 
\begin{assumption}\label{ass2}
There is one residue class modulo $p^n$, represented by
$b\in\mathbb{Z}$ with $0\leq b<p^n$, such that
$b_i\equiv b\bmod p^n$ for
$1\leq i\leq n$.
\end{assumption}
Under this assumption $\eub(a)\equiv ab\bmod p^n$. Furthermore,
$\eua(t)\equiv -b^{-1}t\bmod p^n$.  Restated in terms of upper
ramification numbers, Assumption \ref{ass2} becomes $u_{i+1}\equiv
u_i\bmod p^{n-i}$ for $1\leq i<n$. Since $u_1=b_1\in \mathbb{Z}$, this
implies the conclusion of the Theorem of Hasse-Arf, namely that the
upper ramification numbers are integers. But Assumption \ref{ass2} is
stronger than the conclusion of Hasse-Arf, since it
implies that the upper ramification numbers are integers 
congruent modulo $p$.

Define {\em truncated exponentiation} by
$$X^{[Y]}=\sum_{i=0}^{p-1}\binom{Y}{i}(X-1)^i\in\mathbb{Z}_{(p)}[X,Y]$$
where $\mathbb{Z}_{(p)}$ is the integers localized at $p$.
Motivated by \cite{elder:scaffold}, we define:
\begin{definition} \label{Theta}
Let $\Psi_i=\Theta_i-1$ where
$\Theta_n=\sigma_n$,
and for $1\leq i\leq n-1$,
$$\Theta_i=\sigma_i\Theta_n^{[-\mu_{i,n}]}\Theta_{n-1}^{[-\mu_{i,n-1}]}\cdots
\Theta_{i+1}^{[-\mu_{i,i+1}]}.$$ 
\end{definition}

\begin{remark}
If $K_0$ has characteristic $p$ and $K_n/K_0$ is elementary abelian,
it was observed in \cite{elder:scaffold} that the elements in
Definition \ref{Theta} solve the matrix equation:
$$\begin{pmatrix}
\mu_{1,1} & \mu_{1,2} &\cdots & \mu_{1,n} \\ 
0 & \mu_{2,2} &\cdots & \mu_{2,n} \\ 
& & \ddots &  &\\ 
0 &\cdots &0 & \mu_{n,n}
\end{pmatrix}\cdot \begin{pmatrix}
\Theta_1\\\Theta_2\\\vdots \\\Theta_n\end{pmatrix}=\begin{pmatrix}
\sigma_1\\\sigma_2\\\vdots \\\sigma_n\end{pmatrix}$$ where the usual
vector space operations of addition and scalar multiplication have
been replaced by multiplication and scalar truncated exponentiation,
respectively.  Note $\alpha^p=0$ for all $\alpha$ in the augmentation ideal
$(\sigma-1:\sigma\in G)\subset K_0[G]$.  So, since $\Theta_j-1\in(\sigma-1:\sigma\in G)$ satisfies $(\Theta_j-1)^p=0$,
we find $\Theta_j^{[-\mu_{i,j}]}\cdot \Theta_j^{[\mu_{i,j}]}=1$.  A
cautionary remark is important here: Since scalar truncated
exponentiation does not distribute (it is easy to check for $p=2$ that
the units $(\Theta_i\Theta_j)^{[\mu]}$ and $\Theta_i^{[\mu]}
\Theta_j^{[\mu]}$ are not equal), applying the inverse matrix
$(\mu_{i,j})^{-1}$ to both sides of this equation does not preserve
equality.\end{remark}

The following assumption will enable us to 
ignore the error terms in \eqref{sigmaX}.
\begin{assumption}\label{ass3} Given an integer $\euc\geq 1$, assume that
for $1\leq i\leq j\leq n$,
$$v_n(\epsilon_{i,j})-
v_n(\mu_{i,j})\geq 
(p-1)\sum_{k=1}^{i-1}p^{n-k-1}b_k+(p^{n-i}-p^{n-j})b_i+\euc,$$
which because of (\ref{upper}), is equivalent to $v_n(\epsilon_{i,j})-
v_n(\mu_{i,j})\geq p^{n-1}u_i-p^{n-j}b_i+\euc$.

\end{assumption}

We state the main result of this paper.

\begin{theorem}\label{main}
Let $K_n/K_0$ be a totally ramified Galois $p$-extension with
ramification multiset $\{b_i:1\leq i\leq n\}$ satisfying Assumptions
\ref{ass1} and \ref{ass2}. Thus there is one congruence class modulo
$p^n$, represented by $0\leq b<p^n$, that contains all the
ramification numbers.  Given an integer $\euc\geq 1$, assume that it
is possible to make Choices \ref{choice1} and \ref{choice2} so that
Assumption \ref{ass3} holds.  Then a $K_0[G]$-scaffold on $K_n$, as
defined in \cite[Definition 2.3]{byott:A-scaffold}, exists with
precision $\euc$ and shift parameters $b_1, \ldots, b_n$. Namely,
there are:
\begin{enumerate}
\item[(i)] $\lambda_t\in K_n$ defined in Definition
\ref{lambda}  satisfying $v_n(\lambda_i)=i$ and
$\lambda_{i+p^n}=\pi_0\lambda_i$ for some fixed prime element
$\pi_0\in K_0$. 
\item[(ii)] $\Psi_i\in K_0[G]$ defined in Definition \ref{Theta},
satisfying $\Psi_i1=0$, such that for all $1\leq i\leq n$ and $j\in
\mathbb{Z}$, modulo $\lambda_{j+p^{n-i}b_i}\euP_n^{\euc}$,
$$ \Psi_i \lambda_j \equiv \begin{cases} \lambda_{j+p^{n-i}b_i}&
  \mbox{if } \eua(j)_{(n-i)} \geq 1, \\ 0 & \mbox{if }\eua(j)_{(n-i)}
  =0,\end{cases} $$
where
$\eua$ is the function defined on the integers by $\eua(j)\equiv
-jb^{-1}\bmod p^n$ and $0\leq \eua(j)<p^n$, and $\eua(j)_{(n-i)}$ is
the coefficient of $p^{n-i}$ in the base $p$ expansion of $\eua(j)$.
\end{enumerate}
\end{theorem}

Before we prove this theorem, some discussion of the elements
$\Psi_j\in K_0[G]$ is warranted.  Suppose $\psi_j \in K_0[G]$
satisfies $\mbox{Tr}_{n,j}\psi_j=(\sigma_j-1)\mbox{Tr}_{n,j}$ where
$\mbox{Tr}_{n,j}=\sum_{\sigma\in H_{j}}\sigma$ is the element of
$K[G]$ that gives the trace from $K_n$ to $K_{j}$. In this case, we
will say that $\psi_j$ is a {\em lift} of $(\sigma_j-1)$. Thus
$\Psi_j$ can be considered to be one among many lifts of
$(\sigma_j-1)$.  Now observe that
$v_j((\sigma_j-1)\alpha)-v_j(\alpha)=b_j$ for $\alpha\in K_j$ with
$p\nmid v_j(\alpha)$ and thus
$v_n((\sigma_j-1)\alpha)-v_n(\alpha)=p^{n-j}b_j$.  The following
result states that $p^{n-j}b_j$ is a natural upper bound on
$v_n(\psi_j\alpha)-v_n(\alpha)$ for a lift $\psi_j$ of
$(\sigma_j-1)$. From this perspective, Theorem \ref{main} states that
the lifts $\Psi_j$, provided by Definition \ref{Theta}, are special in
that they achieve a natural upper bound.

\begin{proposition}\label{up-bound}
Let $K_n/K_0$ be a totally ramified Galois $p$-extension
satisfying Assumptions \ref{ass1} and \ref{ass2}.
Let $1 \leq j \leq n$ and let $\psi_j$ be any element of $K_0[G]$ such that
$\mbox{\rm Tr}_{n,j}\psi_j=(\sigma_j-1)\mbox{\rm Tr}_{n,j}$. If 
$\rho\in K_n$ with $v_n(\rho)\equiv b_n\bmod p^{n-j}$ but
$v_n(\rho)\not\equiv b_n(1-p^{n-j})\bmod p^{n-j+1}$ (which is equivalent to
$p\nmid v_j(\mbox{\rm Tr}_{n,j}\rho)$), then
$$v_n(\psi_j\rho)\leq v_n(\rho)+p^{n-j}b_j.$$
\end{proposition}
\begin{proof}
The case $j=n$ is trivial since we necessarily have
$\psi_n=\sigma_n-1$. Fix $j<n$ and consider the different
$\mathfrak{D}_{K_{n}/K_{j}}$ of the extension $K_n/K_j$.  Hilbert's
formula for the exponent of the different \cite[IV\S1 Proposition
  4]{serre:local} gives $\mathfrak{D}_{K_{n}/K_{j}}=\euP_n^m$ where
$m=(b_{j+1}+1)(p^{n-j}-1)+\sum_{i=j+1}^{n-1}(b_{i+1}-b_{i})(p^{n-i}-1)$.
For any $r \in \Z$, we have $\mbox{Tr}_{n,j}(\euP_n^r)=\euP_{j}^{s_r}$
where $s_r=\lfloor(m+r)/p^{n-j}\rfloor$ and $\lfloor x\rfloor$ denotes
the greatest integer $\leq x$.  Since $p^{i+1}\mid (b_{i+1}-b_i)$ by
Assumption \ref{ass2}, it follows that
$$s_r=(b_{j+1}+1)+\sum_{i=j+1}^{n-1}(b_{i+1}-b_{i})p^{j-i}+\left\lfloor\frac{-1-b_n+r}{p^{n-j}}\right\rfloor.$$
In particular, if $r=b_n +kp^{n-j}$ for some $k \in \Z$, we find that 
$s_{r+1}>s_r$ and $s_r \equiv b_{j+1}+k \bmod p$. 
Let $\rho \in K_n$ with $v_n(\rho)=r$. We may write an arbitrary
element $\alpha \in \euP_n^r$ as $\alpha = x \rho + \nu$ with $x \in
\euO_j$ and $\nu \in \euP_n^{r+1}$. Since $s_{r+1}>s_r$, it follows
that $v_j(\Tr_{n,j}\rho)=s_r$, and hence that $v_j((\sigma_j-1)
\Tr_{n.j}\rho)=s_r+b_j$ provided that $k \not \equiv -b_{j+1} \bmod
p$. Recalling Assumption \ref{ass1}, we have therefore shown that if
$v_n(\rho)=r \equiv b_n \bmod p^{n-j}$ but $r \not \equiv
b_n(1-p^{n-j}) \bmod p^{n-j+1}$ then $v_n((\sigma_j-1)\Tr_{n,j} \rho) =
s_r+b_j$.

Now, with $\rho$ as above, suppose that $\psi_j$ is any element of
$K_0[G]$ with $v_n(\psi_j \rho)>r':=r+p^{n-j} b_j$. Since $r' \equiv
b_n \bmod p^{n-j}$, we have $s_{r'+1}>s_{r'}=s_r+b_j$, so that 
$v_j(\Tr_{n,j} \psi_j \rho) > s_r + b_j = v_n((\sigma_j-1)\Tr_{n,j}
\rho)$. Hence $\Tr_{n,j} \psi_j \neq (\sigma_j-1)\Tr_{n,j}$.
\end{proof}

We conclude this section by recording a technical question. 

\begin{question}\label{warning}
A bijection exists between the one-units of $\euO_j$ and the choices
possible in Choice \ref{choice2}.  Namely, given $X_j$ satisfying
Choice \ref{choice2} and any $u_j\in 1+\euP_j$, then $u_jX_j$ will
also satisfy Choice \ref{choice2}.  So how does one optimize the
choice of $X_j$ in Choice \ref{choice2} to maximize the precision
$\euc$ available in Assumption \ref{ass3}?
\end{question}
We do not
address this question here.  Neither was it addressed in
\cite{elder:scaffold,elder:sharp-crit}.  Thus far, in all these cases,
naive choices were made that turned out to be good enough for a
determination of Galois module structure. There has been no need.

\subsection{Proof of Theorem \ref{main}}\label{sec2}
We are interested in analyzing the expression $\Psi_i\lambda_j$
for $1\leq i\leq n$ and $j\in\mathbb{Z}$, where
$\lambda_j$ is as in Definition \ref{lambda} and
$\Psi_i$ is as in Definition \ref{Theta}.
So observe that
$\Psi_i\lambda_j=\pi_0^{f_j}\cdot \Psi_i\rho_{\eua(j)}$ where
$$\Psi_i\rho_{\eua(j)}=\Psi_i\binom{X_n}{a_{(0)}}\binom{X_{n-1}}{a_{(1)}}\cdots
\binom{X_1}{a_{(n-1)}},$$ for $\eua(j)=a=\sum_{i=1}^na_{(n-i)}p^{n-i}$
with $0\leq a_{(n-i)}<p$.  Our analysis is technical.  To motivate it,
we begin by considering the special case treated in
\cite{elder:scaffold} where $\epsilon_{i,j}=0$ for all $1\leq i,j\leq n$.
This gives us the opportunity to more fully
justify \cite[(4)]{elder:scaffold}. Observe that Theorem \ref{main}
with $\euc=\infty$ follows from Proposition \ref{propB}
setting $\kappa_i=0$.

\begin{proposition}\label{propB}
Suppose that Assumptions \ref{ass1} and \ref{ass2} hold, and that
$\epsilon_{i,j}=0$ for all $1\leq i,j\leq n$ so that Assumption
\ref{ass3} holds vacuously. Then for all $0\leq a_{(i)}< p$ and
$\kappa_i\in K_0$,
$$\Psi_j\prod_{i=1}^n\binom{X_i+\kappa_i}{a_{(n-i)}}=\binom{X_j+\kappa_j}{a_{(n-j)}-1}\prod_{i\neq j}\binom{X_i+\kappa_i}{a_{(n-i)}}$$
\end{proposition}
\begin{proof}
Note that $\Theta_j$ fixes $X_i$ for $i<j$. So it is sufficient to prove by inducting down from $j=n$ to $j=1$ that
$\Psi_j\prod_{i=j}^n\binom{X_i+\kappa_i}{a_{(n-i)}}=\binom{X_j+\kappa_j}{a_{(n-j)}-1}\prod_{i=j+1}^n\binom{X_i+\kappa_i}{a_{(n-i)}}$.  
Recall
$\Theta_n=\sigma_n$.  Pascal's Identity states that
$\binom{X_n+\kappa_n}{a_{(0)}-1}+\binom{X_n+\kappa_n}{a_{(0)}}=\binom{X_n+\kappa_n+1}{a_{(0)}}$.
Thus
$\Psi_n\binom{X_n+\kappa_n}{a_{(0)}}=\binom{X_n+\kappa_n}{a_{(0)}-1}$
for $1\leq a_{(0)}< p$. Recall
$\Theta_{n-1}=\sigma_{n-1}\Theta_n^{[-\mu_{n-1,n}]}$.
Observe that
\begin{multline*}
\Theta_n^{[-\mu_{n-1,n}]}\binom{X_n+\kappa_n}{a_{(0)}}\binom{X_{n-1}+\kappa_{n-1}}{a_{(1)}}\\
=\sum_{s=0}^{p-1}\binom{-\mu_{n-1,n}}{s}\Psi_n^s\binom{X_n+\kappa_n}{a_{(0)}}\binom{X_{n-1}+\kappa_{n-1}}{a_{(1)}}\\
=\sum_{s=0}^{a_{(0)}}\binom{-\mu_{n-1,n}}{s}\binom{X_n+\kappa_n}{a_{(0)}-s}\binom{X_{n-1}+\kappa_{n-1}}{a_{(1)}}\\=
\binom{X_n+\kappa_n-\mu_{n-1,n}}{a_{(0)}}\binom{X_{n-1}+\kappa_{n-1}}{a_{(1)}}
\end{multline*}
where the last equality is a consequence of Vandermonde's Convolution
Identity, $\sum_{i=0}^m\binom{a}{i}\binom{b}{m-i}=\binom{a+b}{m}$.
Thus because $\sigma_{n-1}X_n=X_n+\mu_{n-1,n}$ we have
$$\sigma_{n-1}\Theta_n^{[-\mu_{n-1,n}]}\binom{X_n+\kappa_n}{a_{(0)}}\binom{X_{n-1}+\kappa_{n-1}}{a_{(1)}}=\binom{X_n+\kappa_n}{a_{(0)}}\binom{X_{n-1}+\kappa_{n-1}+1}{a_{(1)}}.$$
So
$\Psi_{n-1}\binom{X_n+\kappa_n}{a_{(0)}}\binom{X_{n-1}+\kappa_{n-1}}{a_{(1)}}=
\binom{X_n+\kappa_n}{a_{(0)}}\binom{X_{n-1}+\kappa_{n-1}}{a_{(1)}-1}$,
based upon Pascal's Identity.
Note the role of Pascal's Identity and Vandermonde's Convolution
Identity. These two identities will be used
repeatedly and without mention in the induction step.

Assume that the proposition holds for all $j$ such that $k< j\leq
n$.  Thus for each $j$ with $k< j\leq n$,
\begin{multline} \label{parts}
\Theta_{j}^{[-\mu_{k,j}]}\prod_{i=1}^n\binom{X_i+\kappa_i}{a_{(n-i)}}
=\sum_{s=0}^{p-1}\binom{-\mu_{k,j}}{s}\Psi_j^s
\prod_{i=1}^n\binom{X_i+\kappa_i}{a_{(n-i)}}\\
=\sum_{s=0}^{a_{(n-j)}}\binom{-\mu_{k,j}}{s}\binom{X_j+\kappa_j}{a_{(n-j)}-s}
\prod_{i\neq j}\binom{X_i+\kappa_i}{a_{(n-i)}}
=\binom{X_j+\kappa_j-\mu_{k,j}}{a_{(n-j)}}
\prod_{i\neq j}\binom{X_i+\kappa_i}{a_{(n-i)}}.
\end{multline}
Since $\kappa_j'=\kappa_j-\mu_{k,j}$ is just another element of $K_0$, we find, by applying (\ref{parts}) repeatedly that 
$$\prod_{s=k+1}^{n}\Theta_{s}^{[-\mu_{k,s}]}\cdot \prod_{i=1}^n\binom{X_i+\kappa_i}{a_{(n-i)}}
=
\prod_{i=k+1}^n\binom{X_i+\kappa_i-\mu_{k,i}}{a_{(n-i)}}\cdot 
\prod_{i=1}^k\binom{X_i+\kappa_i}{a_{(n-i)}}.$$
Thus 
$\Theta_k\prod_{i=1}^n\binom{X_i+\kappa_i}{a_{(n-i)}}
=
\prod_{i=k+1}^n\binom{X_i+\kappa_i}{a_{(n-i)}}\cdot 
\binom{X_k+\kappa_k+1}{a_{(n-k)}}
\cdot
\prod_{i=1}^{k-1}\binom{X_i+\kappa_i}{a_{(n-i)}}$,
which means that for $0\leq a_{(n-k)}<p$,
$\Psi_k\prod_{i=1}^n\binom{X_i+\kappa_i}{a_{(n-i)}}=
\binom{X_k+\kappa_k}{a_{(n-k)}-1}
\prod_{i\neq k}\binom{X_i+\kappa_i}{a_{(n-i)}}$.
\end{proof}

\subsubsection{Preliminary results for Theorem \ref{main}}

For $1\leq r\leq s\leq n$, set
$$M_r^s=\prod_{i=r}^{s-1}\mu_{i,i+1},$$
and define an ideal $I_r$ of $\euO_n$ by its generators:
$$I_r=\biggl (M_r^s\cdot \epsilon_{s,t}\cdot X_t^{-1}:r\leq s\leq t\leq  n\biggr ).$$
\begin{lemma}\label{I_r}
$$I_r=
\sum_{i=r}^{n}\epsilon_{r,i}X_{i}^{-1}\euO_n+
\sum_{s=0}^{n-r-1}
\mu_{r,n-s}I_{n-s}.$$
\end{lemma}
\begin{proof}
We can partition the generators of $I_r$ into those with $r=s$ and those with $r<s$. When $r=s$ we have $M_r^s=1$.
When $r<s$ we have $M_r^s=\mu_{r,r+1}M_{r+1}^s$.
As a result,
$$I_r=\sum_{i=r}^n\epsilon_{r,i}X_i^{-1}\euO_n+\mu_{r,r+1}I_{r+1}.$$

Given $i-1\geq r$ we have $b_{i-1}\geq b_r$ and thus
$v_n(\mu_{r,i})\geq v_n(\mu_{r,i-1}\mu_{i-1,i})$. This means that
$\mu_{r,i}\cdot M_i^s\epsilon_{s,t}X_t^{-1}\euO_n\subseteq 
\mu_{r,i-1}\cdot M_{i-1}^s\epsilon_{s,t}X_t^{-1}\euO_n$,
and so $\mu_{r,i}I_i\subseteq \mu_{r,i-1}I_{i-1}$. As a result,
for $1\leq r\leq n$, we have
$\mu_{r,n}I_n\subseteq \mu_{r,n-1}I_{n-1}\subseteq \cdots \subseteq \mu_{r,r+1}I_{r+1}$, and thus
$$\sum_{s=0}^{n-r-1}
\mu_{r,n-s}I_{n-s}=\mu_{r,r+1}I_{r+1}.$$
\end{proof}
\begin{lemma}\label{oops}
$I_r\subseteq X_r^{-1}\euP_n^{\euc}$ if and only if for all $1\leq r\leq s\leq t\leq n$,
$$v_n(\epsilon_{s,t})-v_n(\mu_{s,t})\geq (p-1)\sum_{i=r}^{s-1}p^{n-i-1}b_i+(p^{n-s}-p^{n-t})b_s+\euc.$$
\end{lemma}
\begin{proof}
Observe that $v_n(M_r^s\epsilon_{s,t}X_t^{-1})\geq v_n(X_r^{-1}\pi_n^{\euc})$ is
equivalent to
$\sum_{i=r}^{s-1}p^{n-i-1}(b_i-b_{i+1})+v_n(\epsilon_{s,t})+p^{n-t}b_t\geq
p^{n-r}b_r+\euc$, and that this is equivalent to
$v_n(\epsilon_{s,t})-v_n(\mu_{s,t})\geq
p^{n-r}b_r-p^{n-t}b_t-\sum_{i=r}^{s-1}p^{n-i-1}(b_i-b_{i+1})-p^{n-t}(b_s-b_t)
+\euc=(p-1)\sum_{i=r}^{s-1}p^{n-i-1}b_i+(p^{n-s}-p^{n-t})b_s+\euc$.
\end{proof}
\begin{corollary}\label{assump3}
Assumption \ref{ass3} holds with precision $\euc$ if and only if
$I_r\subseteq X_r^{-1}\euP_n^{\euc}$ for all $1\leq
r\leq n$.
\end{corollary}

\subsubsection{Main result for Theorem \ref{main}}
Since $\Theta_j$ fixes $X_i$ for $i<j$, Theorem \ref{main} follows from
Corollary \ref{assump3} and Proposition \ref{prop-main} below by
specializing to the case $\kappa_i=0$.

\begin{proposition}\label{prop-main}
Suppose that Assumptions \ref{ass1} and \ref{ass2} hold, and
that Assumption \ref{ass3} holds with precision $\euc\geq 1$. Then
for all $0\leq a_{(i)}< p$ and any $\kappa_i\in K_0$ with $v_i(X_i)<v_i(\kappa_i)$,
$$\Psi_j\prod_{i=j}^n\binom{X_i+\kappa_i}{a_{(n-i)}}\equiv
\binom{X_j+\kappa_j}{a_{(n-j)}-1}\cdot \prod_{i=j+1}^n\binom{X_i+\kappa_i}{a_{(n-i)}}\bmod
\prod_{i=j}^n\binom{X_i}{a_{(n-i)}}\cdot I_j.$$
\end{proposition}
\begin{proof}
We induct down from $j=n$ to $j=1$.  Note $\Theta_n=\sigma_n$, and
observe that
$$\sigma_n\binom{X_n+\kappa_n}{a_{(0)}}=\binom{X_n+\kappa_n+1+\epsilon_{n,n}}{a_{(0)}}\equiv \binom{X_n+\kappa_n+1}{a_{(0)}}
\bmod
\epsilon_{n,n}X_n^{-1}\binom{X_n}{a_{(0)}}.$$
Using Pascal's Identity and the definition of $I_n$, this means
$$(\sigma_n-1)\binom{X_n+\kappa_n}{a_{(0)}}\equiv 
\binom{X_n+\kappa_n}{a_{(0)}-1}
\bmod 
\binom{X_n}{a_{(0)}}\cdot I_n.$$
We have proven the case $j=n$.

Assume that Proposition \ref{prop-main} holds for all $j$ with $k<j\leq n$. We aim
to prove that it continues to hold for $j=k$.
Since
$$\left\{\prod_{i=j}^n\binom{X_j+\kappa_i}{a_{(n-i)}}:0\leq
a_{(n-i)}<p\right\}$$ is a basis for $K_n$ over $K_{j-1}$, we can
express any element of $K_n$ in terms of this basis. Our assumption that
Proposition \ref{prop-main} holds for $k<j\leq n$,
together with Corollary \ref{assump3}, means that
$v_n((\Theta_j-1)\alpha)\geq v_n(\alpha)+p^{n-j}b_j$ for all
$\alpha\in K_n$.  As a result, we see that for $1\leq s\leq p-1$, 
\begin{multline*}
v_n\left (\binom{-\mu_{k,j}}{s}(\Theta_{j}-1)^{s-1}I_{j}\prod_{i=j}^n\binom{X_{i}+\kappa_i}{a_{(n-i)}}\right)\geq \\sp^{n-j}(b_k-b_{j})+(s-1)p^{n-j}b_{j}+v_n(I_{j})+v_n\left (\prod_{i=j}^n\binom{X_{i}+\kappa_i}{a_{(n-i)}}\right ).
\end{multline*}
Note the right-hand side is minimized by $s=1$. As a result, using
Proposition \ref{prop-main} for each $k<j$, we have
\begin{multline*}
\Theta_{j}^{[-\mu_{k,j}]}\prod_{i=j}^n\binom{X_{i}+\kappa_i}{a_{(n-i)}}=
\sum_{s=0}^{p-1}\binom{-\mu_{k,j}}{s} \Psi_{j}^{s}
\prod_{i=j}^n\binom{X_{i}+\kappa_i}{a_{(n-i)}}\\
\equiv \prod_{i=j+1}^n\binom{X_{i}+\kappa_i}{a_{(n-i)}}\sum_{s=0}^{a_{(n-j)}}\binom{-\mu_{k,j}}{s}\binom{X_{j}+\kappa_{j}}{a_{(n-j)}-s}
\bmod \binom{-\mu_{k,j}}{1}
I_{j}
\prod_{i=j}^n\binom{X_{i}}{a_{(n-i)}}.
\end{multline*}
Vandermonde's Convolution Identity yields, 
modulo $\mu_{k,j} I_{j} \prod_{i=j}^n\binom{X_{i}}{a_{(n-i)}}$,
$$\Theta_{j}^{[-\mu_{k,j}]}\prod_{i=j}^n\binom{X_{i}+\kappa_i}{a_{(n-i)}} \equiv
\prod_{i=j+1}^n\binom{X_{i}+\kappa_i}{a_{(n-i)}}\cdot
\binom{X_{j}+\kappa_{j}-\mu_{k,j}}{a_{(n-j)}},$$
which holds for all $k<j\leq n$.
Since $\Theta_j$ fixes $X_i$ for $i<j$, this means that 
\begin{multline}\label{inductA}
\Theta_{j}^{[-\mu_{k,j}]}\prod_{i=k}^n\binom{X_{i}+\kappa_i}{a_{(n-i)}} \equiv\\
\prod_{i=j+1}^n\binom{X_{i}+\kappa_i}{a_{(n-i)}}\cdot
\binom{X_{j}+\kappa_{j}-\mu_{k,j}}{a_{(n-j)}-i}\cdot
\prod_{i=k}^{j-1}\binom{X_{i}+\kappa_i}{a_{(n-i)}} \\ \bmod
\mu_{k,j} I_{j} \prod_{i=k}^n\binom{X_{i}}{a_{(n-i)}}.
\end{multline}

Note that, in general, we may consider $\kappa_i'=\kappa_{i}-\mu_{k,i}$ to be
another $\kappa_{i}$. As a result, by repeated use of (\ref{inductA}),
once for each value of $j$ in $k<j\leq n$, we find that
\begin{multline*}
\prod_{j=k+1}^n\Theta_j^{[-\mu_{k,j}]}
\cdot \prod_{i=k}^n\binom{X_{i}+\kappa_{i}}{a_{(n-i)}}
\equiv \prod_{j=k+1}^{n}\binom{X_{j}+\kappa_{j}-\mu_{k,j}}{a_{(n-j)}}\cdot \binom{X_{k}+\kappa_{k}}{a_{(n-k)}}\\
\bmod \left (\sum_{j=k+1}^n
\mu_{k,j}I_{j}\right )\prod_{i=k}^n\binom{X_{i}}{a_{(n-i)}}
\end{multline*}
Notice that the order in which we apply these
$\Theta_j^{[-\mu_{k,j}]}$ does not matter. See Remark \ref{non-abel}.
In any case, if we keep the ordering used in Definition \ref{Theta},
we find
\begin{multline*}
\Theta_{k}
\prod_{i=k}^{n}\binom{X_{i}+\kappa_{i}}{a_{(n-i)}}
\equiv \prod_{j=k+1}^{n}\binom{X_{j}+\kappa_{j}+\epsilon_{k,j}}{a_{(n-j)}}
\cdot \binom{X_{k}+\kappa_{k}+1+\epsilon_{k,k}}{a_{(n-k)}}\\
\bmod \left (\sum_{j=k+1}^n
\mu_{k,j}I_{j}\right )\prod_{i=k}^n\binom{X_{i}}{a_{(n-i)}},\\
\equiv \prod_{j=k+1}^{n}\binom{X_{j}+\kappa_{j}}{a_{(n-j)}}
\cdot \binom{X_{k}+\kappa_{k}+1}{a_{(n-k)}}\\
\bmod \left (
\sum_{j=k}^{n}\epsilon_{k,j}X_{j}^{-1}\euO_n+
\sum_{j=k+1}^n
\mu_{k,j}I_{j}\right )\prod_{i=k}^n\binom{X_{i}}{a_{(n-i)}}.
\end{multline*}
Using $\Psi_k=\Theta_k-1$, Lemma \ref{I_r} and Pascal's Identity, 
the result holds for $j=k$.
\end{proof}

\begin{remark}\label{non-abel}
  The proof of Proposition \ref{prop-main} does not depend
upon the ordering of the factors in $\Theta_{j}$.  
Since the Galois group may be nonabelian, this is noteworthy.
\end{remark}

\section{Elementary abelian $p$-extensions with Galois scaffold} 
\label{elem-ab-GS} 

In this section, we determine conditions that are sufficient for a
totally ramified, elementary abelian extension $L/K$ of degree $p^n$
to satisfy Theorem \ref{main} and thus have a  $K[G]$-scaffold  on $L$.
The main result of this section, Theorem \ref{elem-abel},
extends, from characteristic $p$ to characteristic $0$, the main
result of \cite{elder:scaffold}.

\subsection{Cyclic extensions of degree $p$ in characteristic $0$}

\begin{theorem}\label{C_p}
Let $K$ be a characteristic $0$ local field with perfect residue field
of characteristic $p$, and let $L/K$ be a totally ramified cyclic
extension of degree $p$.  Let the ramification number $u$ for $L/K$
be relatively prime to $p$. (Recall the discussion preceding
Assumption \ref{ass1}.) Then the hypotheses of Theorem \ref{main} hold
and there is a Galois scaffold with precision $\euc=pv_K(p)-(p-1)u\geq
1$.
\end{theorem}
\begin{proof}
Since there is only one break in the ramification filtration, the
lower and upper ramification numbers are the same $b=u$.  Since
$\gcd(p,u)=1$, $L=K(x)$ for some $x$ with $-pv_K(p)/(p-1)<v_L(x)=-u<0$
satisfying an Artin-Schreier equation $\wp(x):=x^p-x\in K$ with
$(\sigma-1) x-1=\delta\in\euP_L$ \cite[III \S2.5 Prop]{fesenko}.
Expand $\wp(\sigma x)=\wp(x)$ to find that
$\sum_{k=1}^{p-1}\binom{p}{k}x^k(1+\delta)^{p-k}\equiv \delta\bmod
\delta\euP_L$. Since $v_L(x)=-u<0$, this means that
$v_L((\sigma-1)x-1)=v_L(\delta)=v_L(px^{p-1})=pv_K(p)-(p-1)u$.  In the
notation of \S\ref{suf-sect}, we have $K_1=L$, $K_0=K$, $X_1=x$ and
$\sigma_1=\sigma$ where $(\sigma_1-1)X_1=\mu_{1,1}+\epsilon_{1,1}$
with $\mu_{1,1}=1$ and $\epsilon_{1,1}=\delta$.  The extension
satisfies Assumptions \ref{ass1}, \ref{ass2} and \ref{ass3} with
precision $\euc=v_L(\delta)$.
\end{proof}

\subsection{Elementary abelian $p$-extensions}\label{elem-ab-main} 
Since the description of the extensions requires a few paragraphs to
develop, we introduce the extensions and state the main theorem first.  We leave
the proofs till \S\ref{choice-X} and \S\ref{val-X}.

Let $K$ be a complete local field whose residue field is perfect of
characteristic $p>0$. Let $L/K$  be a totally ramified,
elementary abelian extension of degree $p^n$, $n>1$.
Again we
change notation so that $L/K=K_n/K_0$.  Fix a composition series
$\{H_i\}$ that refines the ramification filtration of the elementary abelian
group $G=\Gal(K_n/K_0)\cong C_p^n$. Thus $\{H_i\}$ yields elements
$\sigma_i\in G$, lower ramification numbers $b_i$, and upper
ramification numbers $u_i$ via \eqref{upper}, as in \S\ref{suf-sect}.
Restrict these upper ramification numbers as follows.
\begin{assumption}\label{elem-ab-1}
$p\nmid u_1$ and $u_i\equiv u_1\bmod p^{n-1}$ for all $1\leq i\leq n$.
\end{assumption}
Our extension now satisfies Assumptions \ref{ass1} and \ref{ass2}, and restrictions are imposed on the Artin-Schreier
generators of the extension: Let $K_{(i)}$ be the subfield that is
fixed by $\langle \sigma_j:j\neq i\rangle$.  Then because $u_i$ is the
ramification number for $K_{(i)}/K_0$ and $p\nmid u_i$, we have
$K_{(i)}=K_0(x_i)$ where $x_i$ satisfies an Artin-Schreier equation
$\wp(x_i)=x_i^p-x_i=\alpha_i\in K_0$, with $v_{(i)}(x_i)=-u_i$ and
$v_{(i)}((\sigma_i-1)x_i-1)>0$ \cite[III \S2.5 Prop]{fesenko}.
Following the proof of Theorem \ref{C_p},
\begin{equation}\label{sigma-x}
v_{(i)}((\sigma_i-1)x_i-1)=pv_0(p)-(p-1)u_i.
\end{equation}

Let $\beta=\alpha_1$. So $v_0(\beta)=-u_1=-b_1$.
Since $v_0(\alpha_i)=-u_i\equiv-u_j= v_0(\alpha_j)\bmod p^{n-1}$ for
all $i,j$, 
there are $\omega_i\in K_0$ with $\omega_1=1$ and $v_0(\omega_n)\leq
v_0(\omega_{n-1})\leq \cdots \leq v_0(\omega_1)=0$, such that
$\omega_i^{p^{n-1}}\beta\equiv \alpha_i\bmod \alpha_i\euP_0$ for $2\leq i\leq n$.
(Here we have used the fact that the residue field of $K_0$ is
perfect.) Thus
\begin{equation}\label{AS-equations}
\wp(x_i)=\omega_i^{p^{n-1}}\beta+\epsilon_i
\end{equation} for some ``error terms''
$\epsilon_i\in K_0$ with $\epsilon_1=0$ and $v_0(\epsilon_i)> -u_i$.
Note that whenever
$v_0(\omega_i)=v_0(\omega_{i+1})=\cdots=v_0(\omega_j)$ with $i<j$,
$K_0(x_i,\ldots, x_j)/K_0$ has only one ramification number
$u_i=u_j$. As a result, the projections of $\omega_j, \ldots
,\omega_i$ into $\omega_i\euO_0/\omega_i\euP_0$ must be linearly
independent over $\mathbb{F}_p$, the finite field with $p$ elements.
Conversely, given any $\beta$, $\omega_i$, $\epsilon_i$ as above, and
$x_i$ satisfying \eqref{AS-equations},
$K_n=K_0(x_1,x_2,\ldots ,x_n)$ will be a totally ramified elementary abelian
extension of degree $p^n$ with upper ramification numbers $\{u_j:1\leq
j\leq n\}$ satisfying
Assumption \ref{elem-ab-1}.

We now need to subject our extension $K_n/K_0$ to two further restrictions:
First, we ask that the error terms be negligible. Second, we ask that the absolute ramification be relatively large.
To make this precise,
we need further notation:
Let $m_i=v_0(\omega_{i-1})-v_0(\omega_i)\geq 0$ for $i\geq 2$. So
$v_0(\omega_i)=-\sum_{k=2}^im_k$. Note 
$u_i=b_1+p^{n-1}\sum_{k=2}^im_k$ and
$b_i=b_1+p^n\sum_{k=2}^ip^{k-2}m_k$.
Set $C_0=0$, and for $1\leq i\leq n$, 
define $$C_i=u_i-b_i/p^i.$$
Check that $C_i=u_{i+1}-b_{i+1}/p^i$ for $0\leq i\leq n-1$.
Since $C_i=u_{i+1}-b_{i+1}/p^i<u_{i+1}-b_{i+1}/p^{i+1}=C_{i+1}$, the sequence $\{C_i:0\leq i\leq n\}$ is increasing.
The two further restrictions are:
\begin{assumption}
\label{ass-epsilon}  
$v_0(\epsilon_i)>-u_i+C_{n-1}$.
\end{assumption}
\begin{assumption}\label{strong-ass-p} 
$v_0(p)\geq C_n+\euc/p^n$ with $\euc\geq 1$.
\end{assumption}
These assumptions enable us to prove:
\begin{theorem}\label{elem-abel}
Let $K_0$ be a complete local field whose residue field is perfect of
characteristic $p>0$.  Let $K_n/K_0$ be a totally ramified, elementary
abelian extension of degree $p^n$, $n>1$ that satisfies Assumptions
\ref{elem-ab-1}, \ref{ass-epsilon}, \ref{strong-ass-p} with $\euc\geq
1$, and has ramification multiset $\{b_1,b_2,\ldots , b_n\}$. Then the
hypotheses of Theorem \ref{main} hold and we have a scaffold for the
$K_0[G]$-action on $K_n/K_0$ of precision $\euc$, with shift parameters
$b_1,b_2,\ldots , b_n$.
\end{theorem}

To prove this theorem we must, in the notation of \S\ref{suf-sect},
choose elements $X_j\in K_j$ with $v_j(X_j)=-b_j$, as required for
Choice \ref{choice2} so that the difference
$v_j(\epsilon_{i,j})-v_j(\mu_{i,j})$, where
$(\sigma_i-1)X_j=\mu_{i,j}+\epsilon_{i,j}$ as in \eqref{sigmaX},
satisfies Assumption \ref{ass3} with precision $\euc\geq 1$. We define the
$X_j$ in \S\ref{choice-X}. Interestingly, if we assume
$v_j(X_j)=-b_j$, the proof that Assumption \ref{ass3} is satisfied
with precision $\euc\geq 1$ is relatively easy. It appears in
\S\ref{choice-X}, as Lemma \ref{assum3}. The proof that
$v_j(X_j)=-b_j$ is much more involved and appears afterwards, in
\S\ref{val-X}.

\begin{remark}\label{improvements}
In characteristic $p$, Assumption \ref{strong-ass-p} is vacuous, which
is why it did not appear in \cite{elder:scaffold}. Otherwise,
everything in \S\ref{elem-ab-main} is consistent with
\cite{elder:scaffold}. Indeed, Assumption \ref{ass-epsilon} is, after
small changes in notation, exactly \cite[(5)]{elder:scaffold}.
\end{remark}

\subsection{Candidate for Choice \ref{choice2}}
\label{choice-X}
Let $\Omega_{1,j}=\omega_j$, $X_{1,j}=x_j$.
For $2\leq i\leq j\leq n$, recursively define
\begin{eqnarray}
\Omega_{i,j}&=&\wp(\Omega_{i-1,j})/\wp(\Omega_{i-1,i})\in K_0; \mbox{
  in particular, } \Omega_{j,j}=1 \mbox{ for all } j; \label{Omega} \\
X_{i,j}&=&X_{i-1,j}-\Omega_{i-1,j}^{p^{n-i}}X_{i-1,i-1}\in K_{i-1}(x_j)\subseteq K_j; \label{X}
\end{eqnarray}
The following result proves that the $\Omega_{i,j}$, and thus the $X_{i,j}$, are well-defined.
\begin{lemma} \label{v-Omega}
For $2\leq i\leq n$ we have $\wp(\Omega_{i-1,i}) \neq 0$. Furthermore, for $1\leq i\leq j\leq n$
$$v_0(\Omega_{i,j})=-p^{i-1}\sum_{k=i+1}^jm_k=p^{i-n}(u_i-u_j).$$
\end{lemma}
\begin{proof}
To obtain the first assertion, we show that for $1 \leq i \leq n$ we
have $v_0(\Omega_{i,n})\leq v_0(\Omega_{i,n-1})\leq \cdots\leq
v_0(\Omega_{i,i})=0$, and that if $v_0(\Omega_{i,j})=0$ for some $j>i$,
then the projections $\overline{\Omega}_{i,j }, \ldots ,
\overline{\Omega}_{i,i}$ of $\Omega_{i,j }, \ldots , \Omega_{i,i}$ in
$\euO_n/\euP_n$ are linearly independent over $\bF_p$. These assertions
hold for $i=1$ since $\Omega_{1,j}=\omega_j$ with $\omega_1=1$. Assume
inductively that they hold for $i=k-1 \geq 1$. Since
$\Omega_{k-1,k-1}=1$, either $v_0(\Omega_{k-1,k})<0$, or
$v_0(\Omega_{k-1,k})=0$ with
$\overline{\Omega}_{k-1,k}\not\in\mathbb{F}_p$. In either case,
$\wp(\Omega_{k-1,k})\neq 0$, and indeed
$v_0(\wp(\Omega_{j,k}))=pv_0(\Omega_{j,k})$ for all $j<k$.  Furthermore, if
$v_0(\Omega_{k-1,j})=0$ with $j>k-1$ then $v_0(\wp(\Omega_{k-1,j})) =
\cdots = v_0(\wp(\Omega_{k-1,k}))=0$. Also, 
$\wp(\overline{\Omega}_{k-1,j}), \ldots,
\wp(\overline{\Omega}_{k-1,k})$ are linearly independent over $\bF_p$
because $\overline{\Omega}_{k-1,j}, \ldots, \overline{\Omega}_{k-1,k-1}=1$ are
linearly independent and $\wp$ is $\bF_p$-linear with kernel $\bF_p$.
It then follows from \eqref{Omega} that our
assertions hold for $i=k$. This completes the proof that
$\wp(\Omega_{i-1,i}) \neq 0$ for $2\leq i\leq n$.  The formula for
  $v_0(\Omega_{i,j})$ is then easily verified by induction, using
\eqref{Omega}, the definition of the $m_k$, and the fact that 
$v_0(\wp(\Omega_{j,k}))=pv_0(\Omega_{j,k})$ if $j<k$. 
\end{proof}

Using \eqref{X} repeatedly, we find that
$X_{j,j}=X_{1,j}-\sum_{s=1}^{j-1}\Omega_{s,j}^{p^{n-s-1}}X_{s,s}$, or
$x_j=X_{1,j}=X_{j,j}+\sum_{s=1}^{j-1}\Omega_{s,j}^{p^{n-s-1}}X_{s,s}$.
In other words, we have the
matrix equation
$(X_{1,1},X_{2,2},\ldots, X_{n,n})\cdot (\mathbf{\Omega})=(x_1,x_2,\ldots , x_n)$
with
\begin{equation}\label{matrix-eq}
(\mathbf{\Omega})=\begin{pmatrix}1&\Omega_{1,2}^{p^{n-2}}&\Omega_{1,3}^{p^{n-2}}& \cdots &\Omega_{1,n}^{p^{n-2}}\\
0&1&\Omega_{2,3}^{p^{n-3}}& \cdots &\Omega_{2,n}^{p^{n-3}}\\
& & \ddots & &\\
0&0&\cdots& 1 &\Omega_{n-1,n}^{p^0}\\
0&0&\cdots& 0 &1
\end{pmatrix}.
\end{equation}
Clearly,
$K_i=K_0(x_1,\ldots, x_i)=K_0(X_{1,1},\ldots ,X_{i,i})$.
In the next section we will prove that $v_j(X_{j,j})=-b_j$ for $1\leq j\leq n$, so that $X_j=X_{j,j}$ provide candidates for
Choice \ref{choice2}. 
But first we derive an important consequence.

\begin{lemma}\label{assum3}
If $v_j(X_{j,j})=-b_j$ for $1\leq j\leq n$,
then we may use $X_j=X_{j,j}$ for Choice \ref{choice2}. If we do so, then Assumption 
\ref{strong-ass-p} ensures that
Assumption \ref{ass3} holds with precision $\euc\geq 1$.
\end{lemma}
\begin{proof}
Using \eqref{matrix-eq} we find
$((\sigma_i-1)X_{j,j})_{1\leq i,j\leq n}= 
((\sigma_i-1)x_j)_{1\leq i,j\leq n}(\mathbf{\Omega})^{-1}$.
Recall that $(\sigma_i-1)X_{j,j}=0=(\sigma_i-1)x_j$ for $i>j$.
Express the upper triangular matrix  $(\mathbf{\Omega})^{-1}=(\alpha_{i,j})$
for some $\alpha_{i,j}\in K_0$. 
So then for $i\leq j$,
$(\sigma_i-1)X_{j,j}=\alpha_{i,j}(\sigma_i-1)x_i$ where $(\sigma_i-1)x_i\in K_{(i)}$ is a 1-unit.
Recall from \eqref{sigma-x} that $v_{(i)}((\sigma_i-1)x_i-1)=pv_0(p)-(p-1)u_i$. Note that $\alpha_{i,i}=1$.

Since $v_j(X_{j,j})=-b_j$, $v_j((\sigma_i-1)X_{j,j})=b_i-b_j$. Since $(\sigma_i-1)x_i$ is a unit,
$v_0(\alpha_{i,j})=(b_i-b_j)/p^j$. Let $\mu_{i,j}=\alpha_{i,j}$ in \eqref{sigmaX}.
Then $\epsilon_{i,j}=\mu_{i,j}((\sigma_i-1)x_i-1)$
and $v_j(\epsilon_{i,j})=v_j(\mu_{i,j})+
p^{j-1}(pv_0(p)-(p-1)u_i)$, which is equivalent to
$v_n(\epsilon_{i,j})-v_n(\mu_{i,j})=
p^nv_0(p)-(p-1)p^{n-1}u_i$.

Recall that
Assumption \ref{ass3} with precision $\euc\geq 1$, is equivalent to 
$v_n(\epsilon_{i,j})-v_n(\mu_{i,j})\geq p^{n-1}u_i-p^{n-j}b_i+\euc$.
So Assumption \ref{ass3} follows from
$p^nv_0(p)-(p-1)p^{n-1}u_i\geq p^{n-1}u_i-p^{n-j}b_i+\euc$, or
$v_0(p)\geq u_i -b_i/p^{-j}+\euc/p^n$ for all $1\leq i\leq j\leq n$.
Since $C_i=u_i -b_i/p^{-i}\geq u_i -b_i/p^{-j}$ and $\{C_i\}$ is an increasing sequence, this means that
Assumption \ref{ass3} with precision $\euc\geq 1$ follows from
Assumption 
\ref{strong-ass-p}.
\end{proof}

\subsection{Candidate has correct valuation}\label{val-X}
First we define some auxiliary elements. 
Let $B_1=\beta$ and $E_{1,j}=\epsilon_j$, and for $2\leq i\leq j\leq n$ define
\begin{eqnarray}
B_i&=&\wp(X_{i,i}); \label{B-old} \\
E_{i,j}&=&\wp(X_{i,j})- \Omega_{i,j}^{p^{n-i}}B_i. \label{A-S}\\
M_{i-1,j}&=&X_{i,j}^p
-X_{i-1,j}^p+\Omega_{i-1,j}^{p^{n-i+1}}X_{i-1,i-1}^p\in K_{i-1}(x_j),\label{M}\\
L_{i-1,j}&=&
\wp\left
(\Omega_{i-1,i}^{p^{n-i}}\right)\Omega_{i,j}^{p^{n-i}}-\wp\left(\Omega_{i-1,j}^{p^{n-i}}\right)\in
K_0.\label{L}
\end{eqnarray}
Observe that
\eqref{AS-equations} together with
\eqref{A-S} can be restated as 
$\wp(X_{i,j})=\Omega_{i,j}^{p^{n-i}}B_i+E_{i,j}$
for $1\leq i\leq j\leq n$. 
Using \eqref{X}, \eqref{A-S} and \eqref{M}, we calculate
\begin{multline}
 \wp(X_{i,j})  =  X_{i,j}^p - X_{i,j}  
    =  \left( M_{i-1,j} + X_{i-1,j}^p -
 \Omega_{i-1,j}^{p^{n-i+1}} X_{i-1,i-1}^p \right) - X_{i,j} \\
  =  M_{i-1,j} + \wp(X_{i-1,j}) - \Omega_{i-1,j}^{p^{n-i+1}}
     X_{i-1,i-1}^p +  \Omega_{i-1,j}^{p^{n-i}} X_{i-1,i-1}  
  \\=  M_{i-1,j} + \left(\Omega_{i-1,j}^{p^{n-i+1}} B_{i-1}+ E_{i-1,j}
       \right)  
    - \Omega_{i-1,j}^{p^{n-i+1}} \left(B_{i-1}+X_{i-1,i-1}
       \right) + \Omega_{i-1,j}^{p^{n-i}} X_{i-1,i-1}  \\
  =  E_{i-1,j} +M_{i-1,j} - \wp(\Omega_{i-1,j}^{p^{n-i}})
       X_{i-1,i-1}. \label{Xij}
\end{multline}
Using \eqref{Xij} with $j=i$, \eqref{B-old} becomes
\begin{equation}\label{B}
B_i=E_{i-1,i}+ M_{i-1,i}-\wp\left (\Omega_{i-1,i}^{p^{n-i}}\right )X_{i-1,i-1}.
\end{equation}
Use \eqref{Xij} to replace $\wp(X_{i,j})$ in \eqref{A-S}, and then use \eqref{B} to replace $B_i$. The result is 
\begin{equation}\label{E}
E_{i,j}=E_{i-1,j}+M_{i-1,j}-\Omega_{i,j}^{p^{n-i}}(E_{i-1,i}+
M_{i-1,i}) +L_{i-1,j}X_{i-1,i-1}
\end{equation}

We now define $\Omega_k^{\pi(i,j)}\in K_0$ for integers $1\leq i\leq
j\leq k\leq n$. Let $(\mathbf{\Omega}^p)$ be the matrix formed by
replacing each coefficient in $(\mathbf{\Omega})$ with its $p$th power.
The $\Omega_k^{\pi(i,j)}\in K_0$ generalize the coefficients that
appear in the inverse of $(\mathbf{\Omega}^p)$.  Given integers $i\leq
j$, let $\pi(i,j)=\{(a_1,a_2,\ldots , a_t): i= a_1<a_2<\cdots
<a_t\leq j \}$ denote the set of increasing integer sequences that begin at $i$ and end at or before $j$.  Given $k \geq
j$, associate to each
sequence $(a_1,a_2,\ldots , a_t)\in\pi(i,j)$ the product
$(-1)^t\Omega_{a_1,a_2}^{p^{n-a_1}}\Omega_{a_2,a_3}^{p^{n-a_2}}\Omega_{a_3,a_4}^{p^{n-a_3}}\cdots
\Omega_{a_t,k}^{p^{n-a_t}}$.  Let
$$\Omega_k^{\pi(i,j)}=\sum_{(a_1,\ldots , a_t)\in\pi(i,j)}
(-1)^t\Omega_{a_1,a_2}^{p^{n-a_1}}\Omega_{a_2,a_3}^{p^{n-a_2}}\Omega_{a_3,a_4}^{p^{n-a_3}}\cdots
\Omega_{a_t,k}^{p^{n-a_1}}.$$
In particular,
\begin{equation} \label{Omega-prod-short}
     \Omega_j^{\pi(i,i)}=-\Omega_{i,j}^{p^{n-i}}.  
\end{equation}
Observe that for $i<j<k$,
\begin{equation}\label{Omega-prod}
\Omega_k^{\pi(i,j)}=\Omega_j^{\pi(i,j-1)}\Omega_k^{\pi(j,j)}+\Omega_k^{\pi(i,j-1)}.
\end{equation}
Furthermore, for $i<j$,
$$\Omega_{j}^{\pi(i,j-1)}=-\left (
\Omega_{i,i+1}^{p^{n-i}}\Omega_{j}^{\pi(i+1,j-1)} +\cdots +
\Omega_{i,j-1}^{p^{n-i}}\Omega_{j}^{\pi(j-1,j-1)}+\Omega_{i,j}^{p^{n-i}}
\right ),$$ which can be seen as the dot product of the $i$th row of
$(\mathbf{\Omega}^p)$ and the $j$th column of
$$(\mathbf{\Omega}^p)^{-1}=\begin{pmatrix}1&\Omega_2^{\pi(1,1)}&\Omega_3^{\pi(1,2)}& \cdots &\Omega_n^{\pi(1,n-1)}\\
0&1&\Omega_3^{\pi(2,2)}& \cdots &\Omega_n^{\pi(2,n-1)}\\
& & \ddots & &\\
0&0&\cdots& 1 &\Omega_n^{\pi(n-1,n-1)}\\
0&0&\cdots& 0 &1
\end{pmatrix}.$$
Now  check, using Lemma \ref{v-Omega}, that for 
$(a_1,\ldots , a_t)\in\pi(i,j)$
$$v_0\left((-1)^t\Omega_{a_1,a_2}^{p^{n-a_1}}\Omega_{a_2,a_3}^{p^{n-a_2}}\Omega_{a_3,a_4}^{p^{n-a_3}}\cdots
\Omega_{a_t,k}^{p^{n-a_t}}\right)=-p^{n-1}\sum_{s=i+1}^km_s.$$
Thus 
\begin{equation}\label{Omega-bound}
v_0\left(\Omega_k^{\pi(i,j)}\right)\geq -p^{n-1}\sum_{s=i+1}^km_s=u_i-u_k=v_0\left(\Omega_{i,k}^{p^{n-i}}\right).
\end{equation}

\begin{lemma}\label{extended-E}
$E_{i,j}=E_{1,j}+\sum_{s=2}^i \Omega_j^{\pi(s,i)}E_{1,s}+
\sum_{r=1}^{i-1}\left (M_{r,j}+\sum_{s=r+1}^{i}\Omega_j^{\pi(s,i)}M_{r,s}
\right)\\
+\sum_{r=1}^{i-1}\left (L_{r,j}+\sum_{s=r+2}^{i}\Omega_j^{\pi(s,i)}L_{r,s}
\right)X_{r,r}
$ for $1\leq i<j\leq n$.
\end{lemma}
\begin{proof}
This is clear for $i=1$.  For $i=2$ the statement follows directly
from \eqref{E} and \eqref{Omega-prod-short}. 
Assume that the Lemma holds for 
$(i,j)=(i_0-1,j),(i_0-1,i_0)$. 
Using \eqref{Omega-prod-short}, we can write \eqref{E}
as 
$E_{i_0,j}=E_{i_0-1,j}+
\Omega_j^{\pi(i_0,i_0)}E_{i_0-1,i_0}+ 
L_{i_0-1,j}X_{i_0-1,i_0-1}+
M_{i_0-1,j}+
\Omega_j^{\pi(i_0,i_0)}M_{i_0-1,i_0}$.
Using induction, replace $E_{i_0-1,j}$ and $E_{i_0-1,i_0}$. Then
\eqref{E} becomes
\begin{multline*}
E_{i_0,j}=
E_{1,j}+\sum_{s=2}^{i_0-1} \Omega_j^{\pi(s,i_0-1)}E_{1,s}+
\Omega_j^{\pi(i_0,i_0)}\left(
E_{1,i_0}+\sum_{s=2}^{i_0-1} \Omega_{i_0}^{\pi(s,i_0-1)}E_{1,s}\right )\\+
M_{i_0-1,j}+
\sum_{r=1}^{i_0-2}\left (M_{r,j}+\sum_{s=r+1}^{i_0-1}\Omega_j^{\pi(s,i_0-1)}M_{r,s}
\right)\\+\Omega_j^{\pi(i_0,i_0)}M_{i_0-1,i_0}+
\Omega_j^{\pi(i_0,i_0)}
\sum_{r=1}^{i_0-2}\left (M_{r,i_0}+\sum_{s=r+1}^{i_0-1}\Omega_{i_0}^{\pi(s,i_0-1)}M_{r,s}
\right)\\+
L_{i_0-1,j}X_{i_0-1,i_0-1}+
\sum_{r=1}^{i_0-2}\left (L_{r,j}+\sum_{s=r+2}^{i_0-1}\Omega_j^{\pi(s,i_0-1)}L_{r,s}
\right)X_{r,r}\\+
\Omega_j^{\pi(i_0,i_0)}
\sum_{r=1}^{i_0-2}\left (L_{r,i_0}+\sum_{s=r+2}^{i_0-1}\Omega_{i_0}^{\pi(s,i_0-1)}L_{r,s}
\right)X_{r,r}.
\end{multline*}
As a result,
\begin{multline*}
E_{i_0,j}=
E_{1,j}+
\Omega_j^{\pi(i_0,i_0)}E_{1,i_0}+
\sum_{s=2}^{i_0-1} \left(\Omega_j^{\pi(s,i_0-1)}+\Omega_j^{\pi(i_0,i_0)}\Omega_{i_0}^{\pi(s,i_0-1)}\right)E_{1,s}+
\sum_{r=1}^{i_0-1}M_{r,j}
\\+
\Omega_j^{\pi(i_0,i_0)}\sum_{r=1}^{i_0-1}M_{r,i_0}+
\sum_{r=1}^{i_0-2}
\sum_{s=r+1}^{i_0-1}\left(\Omega_j^{\pi(s,i_0-1)}+\Omega_j^{\pi(i_0,i_0)}\Omega_{i_0}^{\pi(s,i_0-1)}\right )M_{r,s}
+
\sum_{r=1}^{i_0-1}L_{r,j}X_{r,r}\\+
\Omega_j^{\pi(i_0,i_0)}
\sum_{r=1}^{i_0-2}L_{r,i_0}X_{r,r}
+
\sum_{r=1}^{i_0-2}\sum_{s=r+2}^{i_0-1}\left (\Omega_j^{\pi(s,i_0-1)}+\Omega_j^{\pi(i_0,i_0)}\Omega_{i_0}^{\pi(s,i_0-1)}\right)L_{r,s}
X_{r,r}.
\end{multline*}
Using \eqref{Omega-prod} , we find that
\begin{multline*}
E_{i_0,j}=
E_{1,j}+
\Omega_j^{\pi(i_0,i_0)}E_{1,i_0}+
\sum_{s=2}^{i_0-1} 
\Omega_j^{\pi(s,i_0)}
E_{1,s}
\\+\sum_{r=1}^{i_0-1}M_{r,j}+
\sum_{r=1}^{i_0-1}\Omega_j^{\pi(i_0,i_0)}M_{r,i_0}+
\sum_{r=1}^{i_0-2}
\sum_{s=r+1}^{i_0-1}\Omega_j^{\pi(s,i_0)}
M_{r,s}
\\+
\sum_{r=1}^{i_0-1}L_{r,j}X_{r,r}+
\sum_{r=1}^{i_0-2}\Omega_j^{\pi(i_0,i_0)}L_{r,i_0}X_{r,r}
+
\sum_{r=1}^{i_0-2}\sum_{s=r+2}^{i_0-1}\Omega_j^{\pi(s,i_0)}
L_{r,s}
X_{r,r},
\end{multline*}
from which the result for $i=i_0$ follows.
\end{proof}

\begin{lemma}\label{lemma-L}
If $v_0(p)\geq C_n$ then 
$v_0(L_{r,s})\geq v_0(p)+u_r-u_s$ for $1\leq r<s\leq n$.
\end{lemma}

\begin{proof}
The formula for $L_{r,s}$, namely \eqref{L},
compared with \eqref{Omega}
leads to an interest in $
\wp(\Omega_{r,s})^{p^t}- \wp(\Omega_{r,s}^{p^t})$
or
$\Omega_{r,s}^{p^{t+1}}\left((1-y)^{p^t}-1+y^{p^t}\right)$ where $t=n-r-1$ and
$y=\Omega_{r,s}^{1-p}\in \euO_0$.  Note that from Lemma \ref{v-Omega} we have
$v_0(y)=p^{r-1}(p-1)\sum_{v=r+1}^sm_v\geq 0$.  So we begin by proving the
following:

Given any prime $p$, 
integer 
$t\geq 1$ and indeterminate $y$, the polynomial $(1-y)^{p^t}-1+y^{p^t}$ 
is contained in the ideal of the polynomial ring $\mathbb{Z}[y]$:
\begin{equation}\label{ideal}
(p^iy^{p^{t-i}}
: 1\leq i\leq t).
\end{equation}
  Although this can be proven by
induction, we prefer an alternate proof using the Binomial Theorem:
$(1-y)^{p^t}-1+y^{p^t}=\sum_{a=1}^{p^t-1}\binom{p^t}{a}(-y)^{a} +
(-y)^{p^t}+y^{p^t}$.  It is a result of Kummer that the exact power of
$p$ dividing $\binom{a+b}{a}$ is equal to the number of ``carries''
when performing the addition of $a$ and $b$, written in base $p$
\cite[pg 24]{ribenboim}.  Given $i$, we are interested in identifying the
smallest integer exponent $a$ such that $a$ plus $p^t-a$ involves exactly $i$
``carries''. This occurs at $a=p^{t-i}$. Note that
$p^{t-i}$ plus $p^t-p^{t-i}=(p-1)p^{t-1}+(p-1)p^{t-2}+\cdots
+(p-1)p^{t-i}$ is $p^t$.

Return to the situation where $y=\Omega_{r,s}^{1-p}\in \euO_0$ and
\eqref{ideal} is an ideal of $\euO_0$. 
We now prove that under $v_0(p)\geq C_n$, 
this ideal is generated by $py^{p^{t-1}}$.
In other words, we prove that
$v_0(p^iy^{p^{t-i}})=iv_0(p)+p^{t-i}v_0(y)\geq 
v_0(py^{p^{t-1}})$ for $1\leq i\leq t$.
Observe that because $b_1>0$ and $m_k\geq 0$ we have $b_1+p^{n-2}\sum_{k=2}^nm_k\geq b_1/p^n+\sum_{k=2}^np^{k-2}m_k$, which is equivalent to $b_1+u_n/p\geq b_1/p+b_n/p^n$ and thus to $u_n-b_n/p^n\geq (p-1)(u_n-b_1)/p$.
As a consequence, $u_n-b_n/p^n\geq (p-1)(u_n-b_1)/p\geq (p-1)(u_n-b_1)/p^2$. Our assumption, 
$v_0(p)\geq C_n$, therefore means that $v_0(p)\geq (p-1)(u_n-b_1)/p^2$. In other words,
$v_0(p)\geq p^{n-3}(p-1)\sum_{k=2}^nm_k$. Thus 
$v_0(p)\geq p^{n-3}(p-1)\sum_{v=r+1}^sm_v$ for all $1\leq r<s\leq n$, and thus
$v_0(p)\geq p^{n-r-2}v_0(y)$. We have shown that $v_0(p)\geq p^{t-1}v_0(y)$. As a result, for $2\leq i$,
$(i-1)v_0(p)\geq v_0(p)\geq p^{t-1}v_0(y)\geq (1/p-1/p^i)p^tv_0(y)$ from which $v_0(p^iy^{p^{t-i}})\geq 
v_0(py^{p^{t-1}})$ follows.

This implies
\begin{equation}\label{implic-Omega1}
\wp(\Omega_{r,s})^{p^t} -\wp\left
(\Omega_{r,s}^{p^t}\right)\in
p\Omega_{r,s}^{p^{t-1}(p^2-p+1)}\euO_0.
\end{equation}
In particular, by setting $s=r+1$, we see that \eqref{implic-Omega1} implies
\begin{equation}
\label{implic-Omega2}
\frac{\wp(\Omega_{r,r+1}^{p^t} )}{\wp(\Omega_{r,r+1})^{p^t}}
\in 1+ p\Omega_{r,r+1}^{-p^{t-1}(p-1)}
\euO_0.
\end{equation}
Replace $\Omega_{r+1,s}$ in the expression for $L_{r,s}$ using
\eqref{Omega}.  Thus, using
\eqref{implic-Omega2} and \eqref{Omega}, we see that
$L_{r,s}\in
\wp(\Omega_{r,s})^{p^t}
(1+p\Omega_{r,r+1}^{-p^{t-1}(p-1)}\euO_0)
-\wp(\Omega_{r,s}^{p^{t}} )$.
Using \eqref{implic-Omega1},
$$L_{r,s}\in
p\Omega_{r,s}^{p^{t-1}(p^2-p+1)}\euO_0+
p\wp(\Omega_{r,s})^{p^t}\Omega_{r,r+1}^{-p^{t-1}(p-1)}\euO_0.$$
Since 
$v_0(\wp(\Omega_{r,s}))=v_0(\Omega_{r,s}^p)$ and
$v_0(\Omega_{r,s})\leq v_0(\Omega_{r,r+1})$,
$v_0(\wp(\Omega_{r,s})^{p^t}\Omega_{r,r+1}^{-p^{t-1}(p-1)})\leq v_0(\Omega_{r,s}^{p^{t-1}(p^2-p+1)})$.
Therefore
$v_0(L_{r,s})\geq v_0(p\wp(\Omega_{r,s})^{p^t}\Omega_{r,r+1}^{-p^{t-1}(p-1)})$, which implies $v_0(L_{r,s})\geq 
v_0(p)-p^{n-1}\sum_{k=r+2}^sm_k-p^{n-3}(p^2-p+1)m_{r+1}\geq v_0(p)+u_r-u_s$.
\end{proof}

\begin{proposition}\label{v-Xi} Under Assumption \ref{ass-epsilon}, if $v_0(p)>C_n$ then
$v_j(X_{j,j})=-b_j$ for $1\leq j\leq n$.
\end{proposition}
\begin{proof}
We point out to the reader that we do not use Assumption \ref{ass-epsilon} until the last third of the proof
where we verify \eqref{goal-main} for $E_{1,s}$.

Define
$v_i^*(x)=p^{i-n}v_n(x)$ for $x\in K_n$,  so that for $x\in K_i$ we
have $v_i^*(x)=v_i(x)$.  
For $1\leq i\leq j\leq n$, define
\begin{equation}\label{ell}
\euc_{i,j}=-b_i-p^{n+i-2}\sum_{k=i+1}^jm_k=p^{i-1}(u_i-u_j)-b_i=v_{i-1}(\Omega_{i,j}^{p^{n-i}})-b_i.
\end{equation}

Our goal is to prove, by induction on $i$, that
for $1\leq i\leq j\leq n$, we have
\begin{equation}\label{val-B}
v_{i-1}^*(B_i)=\euc_{i,i}=-b_i
\end{equation}
\begin{equation}\label{val-EX}
v_{i-1}^*(E_{i,j})>v_i^*(X_{i,j})=\euc_{i,j}.
\end{equation}
The case $i=1$ is immediate from $B_1=\beta$, $X_{1,j}=x_j$ and $E_{1,j}=\epsilon_j$ since
$v_0(B_1)=v_0(\beta)=-b_1$,
$v_1^*(X_{1,j})=v_{(1)}(x_j)=-u_j=-b_1-p^{n-1}\sum_{k=2}^jm_k=\euc_{1,j}$,
and
$v_0(\epsilon_j)>-u_j$.

Given $2\leq i_0\leq n$,
assume that 
\eqref{val-B} and
\eqref{val-EX} hold for all $i=i_0-1\leq j\leq n$.
We need to prove that \eqref{val-B} and
\eqref{val-EX} hold for $i=i_0$.
Observe that once we have proven
$v_{i_0-1}^*(B_{i_0})=\euc_{i_0,i_0}$ and
for $i_0\leq j\leq n$ that
$v_{{i_0}-1}^*(E_{i_0,j})>\euc_{i_0,j}$,
then it is immediate from \eqref{A-S} and Lemma \ref{v-Omega} that $v_{i_0}^*(X_{i_0,j})=\euc_{i_0,j}$. 
Thus we focus on proving that $v_{i_0-1}^*(B_{i_0})=\euc_{i_0,i_0}$ and
for $i_0\leq j\leq n$ that
$v_{{i_0}-1}^*(E_{i_0,j})>\euc_{i_0,j}$.

Consider the expression for $B_{i_0}$ in \eqref{B}. By induction, $v_{i_0-1}(X_{i_0-1,i_0-1})=-b_{i_0-1}$.
Thus, using Lemma \ref{v-Omega}, we have
$v_{i_0-1}(\wp (\Omega_{i_0-1,i_0}^{p^{n-{i_0}}} )X_{i_0-1,i_0-1})=\euc_{i_0,i_0}=-b_{i_0}$.
Use Lemma \ref{extended-E} to expand $E_{i_0-1,i_0}$ so that the other terms in $B_{i_0}$ are
\begin{multline}\label{goal-B}
E_{i_0-1,i_0}+M_{i_0-1,i_0}=
E_{1,i_0}+\sum_{s=2}^{i_0-1} \Omega_{i_0}^{\pi(s,i_0-1)}E_{1,s}+M_{i_0-1,i_0}\\+
\sum_{r=1}^{i_0-2}\left (M_{r,i_0}+\sum_{s=r+1}^{i_0-1}\Omega_{i_0}^{\pi(s,i_0-1)}M_{r,s}
\right)
+\sum_{r=1}^{i_0-2}\left (L_{r,i_0}+\sum_{s=r+2}^{i_0-1}\Omega_{i_0}^{\pi(s,i_0-1)}L_{r,s}
\right)X_{r,r}.
\end{multline}
We will have proven $v_{i_0-1}^*(B_{i_0})=\euc_{i_0,i_0}$ as soon as we prove that the valuation in $v_{i_0-1}^*$ of each term in the right-hand-side of 
\eqref{goal-B}
exceeds $-b_{i_0}=\euc_{i_0,i_0}$.
Similarly,
$v_{{i_0}-1}^*(E_{i_0,j})>\euc_{i_0,j}$
will follow if each term in the right-hand-side of 
\begin{multline}\label{goal-E}
E_{i_0,j}=E_{1,j}+\sum_{s=2}^{i_0 }\Omega_j^{\pi(s,i_0)}E_{1,s}+
\sum_{r=1}^{i_0-1}\left (M_{r,j}+\sum_{s=r+1}^{i_0}\Omega_j^{\pi(s,i_0)}M_{r,s}
\right)\\
+\sum_{r=1}^{i_0-1}\left (L_{r,j}+\sum_{s=r+2}^{i_0}\Omega_j^{\pi(s,i_0)}L_{r,s}
\right)X_{r,r}
\end{multline}
has valuation in $v_{i_0-1}^*$ that exceeds $\euc_{i_0,j}$.
We claim that both of these statements follow if, for $1\leq r<i_0,s\leq n$, we prove that
\begin{multline}\label{goal-main} 
v_{i_0-1}^*(E_{1,s}), v_{i_0-1}^*(M_{r,s}), v_{i_0-1}^*(L_{r,s}X_{r,r})> \\
p^{i_0-1}(u_{i_0}-u_s)-b_{i_0}=
-b_{i_0}
+p^{n+i_0-2}\begin{cases}
\sum_{k=s+1}^{i_0}m_k&\mbox{if }s\leq i_0,\\
-\sum_{k=i_0+1}^sm_k &\mbox{if }i_0<s.
\end{cases}\end{multline}
To prove this claim we begin by noticing that the terms $E_{1,s}$, $M_{r,s}$, $L_{r,s}X_{r,r}$ with 
$s=j$ and $1\leq r<i_0<j\leq n$ only appear in \eqref{goal-E}.
The fact that the valuation in $v_{i_0-1}^*$ of these terms exceeds $\euc_{i_0,j}=\euc_{i_0,s}=-b_{i_0}-p^{n+i_0-2}\sum_{k=i_0+1}^sm_k$
is equivalent to \eqref{goal-main} for $i_0<s$.
The other terms, $E_{1,s}$ with $1< s\leq i_0$, $M_{r,s}$ with $1\leq r< s\leq i_0$, $L_{r,s}X_{r,r}$ 
with $1\leq r<r+2\leq s\leq i_0$,
appear in both \eqref{goal-B} and \eqref{goal-E}. So that we can treat these terms uniformly,
let $T_s$ with $s\leq i_0$ denote one such term (either $E_{1,s}$, $M_{r,s}$ or $L_{r,s}X_{r,r}$), and
notice that \eqref{goal-main} concerning $v_{i_0-1}^*(T_s)$
can be rewritten using Lemma \ref{v-Omega} as
\begin{equation}\label{ggoal}
v_{i_0-1}^*(\Omega_{s,j}^{p^{n-s}}T_s)>\euc_{i_0,j},
\end{equation}
where $j$ is any integer $i_0\leq j\leq n$.
Let $T_s$ be a term in \eqref{goal-B}. We treat the two cases, $s=i_0$ and $s<i_0$, separately.
If $s=i_0$, we need $v_{i_0-1}^*(T_s)>\euc_{i_0,i_0}$, which since $\Omega_{i_0,i_0}^{p^{n-i_0}}=1$ is equivalent to
\eqref{ggoal} with $j=i_0$. 
If $s<i_0$, we need $v_{i_0-1}^*(\Omega_{i_0}^{\pi(s,i_0-1)}T_s)>\euc_{i_0,i_0}$.
This follows from \eqref{ggoal}
with $j=i_0$, using \eqref{Omega-bound}, namely
$v_{i_0-1}^*(\Omega_{i_0}^{\pi(s,i_0-1)}T_s)\geq v_{i_0-1}^*(\Omega_{s,i_0}^{p^{n-s}}T_s)$. 
Now let $T_s$ be a term in \eqref{goal-E}, with $s\leq i_0$. We need $v_{i_0-1}^*(\Omega_{j}^{\pi(s,i_0)}T_s)>\euc_{i_0,j}$.
This follows from \eqref{ggoal}, again using \eqref{Omega-bound},
$v_{i_0-1}^*(\Omega_{j}^{\pi(s,i_0)}T_s)\geq v_{i_0-1}^*(\Omega_{s,j}^{p^{n-s}}T_s)$.

Now we prove, for each of $E_{1,s}$, $M_{r,s}$ and $L_{r,s}X_{r,r}$, that the inequalities in \eqref{goal-main} hold.
Consider \eqref{goal-main} for $E_{1,s}$.
Since
$\{C_i\}$ is an increasing sequence,
$C_{i_0-1}\leq C_{n-1}$. So, using 
Assumption 
\ref{ass-epsilon}, we have $v_0(E_{1,s})>C_{n-1}-u_s\geq C_{i_0-1}-u_s=u_{i_0}-b_{i_0}/p^{i_0-1}-u_s$, which is equivalent to \eqref{goal-main}.

Consider \eqref{goal-main} for $M_{r,s}$, namely $v_0^*(M_{r,s})>(u_{i_0}-u_s)-b_{i_0}/p^{i_0-1}$.
Since $v_0(p)>C_{n-1}$ and $\{C_i\}$ is an increasing sequence, $v_0(p)>C_{i_0-1}-C_{r-1}$.
This means that $v_0(p)+(u_r-u_s)-b_r/p^{r-1}>(u_{i_0}-u_s)-b_{i_0}/p^{i_0-1}$.
Therefore, it is sufficient to prove
$v_0^*(M_{r,s})\geq v_0(p)+(u_r-u_s)-b_r/p^{r-1}$, or equivalently
$v_r^*(M_{r,s})\geq v_r(p)+p^r(u_r-u_s)-pb_r=v_r(p)+p\euc_{r,s}$. 
By induction $v_r^*(X_{r,s})=\euc_{r,s}$ for $1\leq r< i_0$ and all $r\leq s\leq n$.
Thus it is sufficient to prove $v_r^*(M_{r,s})\geq v_r^*(pX_{r,s}^p)=v_r^*(p\Omega_{r,s}^{p^{n-r}}X_{r,r}^p)$.
Recall using \eqref{M} that 
$M_{r,s}=X_{r+1,s}^p-X_{r,s}^p+\Omega_{r,s}^{p^{n-r}}X_{r,r}^p$. Use \eqref{X}
to replace $X_{r+1,s}$. As a result,
$$
M_{r,s}=\sum_{i=1}^{p-1}\binom{p}{i}X_{r,s}^i\left (-\Omega_{r,s}^{p^{n-r-1}}X_{r,r}\right)^{p-i}+
\left (\left(-\Omega_{r,s}^{p^{n-r-1}}\right)^p
+\Omega_{r,s}^{p^{n-r}}\right)X_{r,r}^p.$$
It is sufficient to prove that each nonzero term in this sum has valuation $v_r^*(pX_{r,s}^p)$.
So note
$$((-\Omega_{r,s}^{p^{n-r-1}})^p
+\Omega_{r,s}^{p^{n-r}})X_{r,r}^p=\begin{cases}p\Omega_{r,s}^{p^{n-r}}X_{r,r}^p&\mbox{for }p=2,\\
0&\mbox{for }p>2.\end{cases}$$ 
Furthermore, 
$v_r^*(
\binom{p}{i}X_{r,s}^i(-\Omega_{r,s}^{p^{n-r-1}}X_{r,r})^{p-i})=v_r^*(pX_{r,s}^p)$ for $1\leq i\leq p$,
since we have $v_r^*(X_{r,s})=v_r^*(\Omega_{r,s}^{p^{n-r-1}}X_{r,r})$.

Consider \eqref{goal-main} for $L_{r,s}X_{r,r}$, namely $v_0^*(L_{r,s}X_{r,r})> (u_{i_0}-u_s)-b_{i_0}/p^{i_0-1}$.
From Lemma \ref{lemma-L}, assuming $v_0(p)> C_n$, we have $v_0(L_{r,s})\geq v_0(p)+u_r-u_s$.
Since $v_0(p)> C_{i_0-1}-C_r$, this means that $v_0(L_{r,s})> C_{i_0-1}-C_r+u_r-u_s=(u_{i_0}-u_s)-b_{i_0}/p^{i_0-1}+b_r/p^r$.
Since $r<i_0$, $v_r^*(X_{r,r})=-b_r$, and so $v_0^*(L_{r,s}X_{r,r})> (u_{i_0}-u_s)-b_{i_0}/p^{i_0-1}$.
\end{proof}
Proposition \ref{v-Xi} completes the proof of Theorem \ref{elem-abel}.

\section{Elementary abelian examples and explicit Galois module structure}
\label{examples-sect}
In this section, we illustrate the explicit nature of what is possible
when one combines the results of this paper with those of
\cite{byott:A-scaffold}. We choose to do so in the context of two
classes of totally ramified extensions, biquadratic and weakly
ramified, that have a long history and for which explicit results
already exist.  Furthermore, since it can be done quickly, we also
apply our results to V.~Abrashkin's ``elementary extensions''
\cite{abrashkin}. All these results are in characteristic $0$.  For
analogous results in characteristic $p$, see
\cite[\S4]{byott:A-scaffold}.

\subsection{Biquadratic extensions} \label{biquadratic}

Let $K_0$ be a local field of characteristic $0$ with perfect residue
field of characteristic $2$, and let $K_2$ be a totally ramified
Galois extension of $K_0$ with $G=\mbox{Gal}(K_2/K_0)\cong
C_2\times C_2$. The structure of $\euO_2$ over its associated order
$\euA_{K_2/K_0}$ in $K_0[G]$ was investigated by B.~Martel
\cite{martel}. Here, we use \cite[Thm 3.1]{byott:A-scaffold} to
recover a large part of Martel's result, but also extend his result to
arbitrary ideals $\euP_2^h$. Exclude 
the case where $K_2/K_0$ contains a maximally ramified
quadratic subextension. (Martel's results include this case, and also
the case where $K_2/K_0$ is not totally ramified.)  Then the upper
ramification numbers $u_1 \leq u_2$ are both odd, and the lower
ramification numbers $b_1 \leq b_2$ are congruent modulo $4$. We then
have $2b_1+b_2=u_1+2u_2 \leq 6v_0(2)-3$.
Define
$$\euA_h=\{\alpha\in K_0[G]:\alpha\euP_2^h\subseteq\euP_2^h\}.$$
Note that $\euA_0=\euA_{K_2/K_0}$.
Martel's result is 
that $\euO_2$ is free over $\euA_0$ if and only if
\begin{equation} \label{martel}
  2b_1+b_2 \leq 4v_0(2) + 3(-1)^{(b_1-1)/2}. 
\end{equation}

In other words, \cite{martel} finds that
$\euO_2$ is always free over $\euA_0$ when
$v_0(2)$ is sufficiently large relative to $b_1$ and $b_2$.  In
Proposition \ref{biquad-prop} below, we find, also assuming $v_0(2)$
is sufficiently large, that $\euP_2^3$ is always free over $\euA_3$,
that $\euP_2$ is free over $\euA_1$ if and only if $b_1\equiv 1\bmod 4$,
and that $\euP_2^2$ is free over $\euA_2$ if and only if $b_1\equiv
3\bmod 4$. In each case, what we mean by ``sufficiently large'' is determined by
Proposition \ref{biq-l-0}
and the value of
$h$ in $\euP_2^h$. When $h=0$, our result excludes only one case covered
by \eqref{martel}, namely the case $2b_1+b_2 =4v_0(2)+3$ when $b_1
\equiv 1 \bmod 4$.

\begin{proposition} \label{biq-l-0}
Let $K_2/K_0$ be a totally ramified biquadratic extension in
characteristic $0$ whose lower ramification numbers satisfy 
$2b_1+b_2 < 4v_0(2)$. Then $K_2/K_0$ has a Galois scaffold of precision
$\euc=4v_0(2)-2b_1-b_2 \geq 1$.  
\end{proposition}
\begin{proof}
The condition $2b_1+b_2 < 4v_0(2)$ ensures that $u_2 < 2 v_0(2)$, so
that $u_1$, $u_2$ are indeed odd. We have $b_1=u_1$, $b_2=b_1+4m$,
$u_2=u_1+2m$ for some integer $m \geq 0$. Then $K_2=K_0(x_1,x_2)$ with
$\wp(x_1)=\beta\in K_0$ where $v_0(\beta)=-b_1$ and
$\wp(x_2)=\omega^2\beta+\epsilon$ where $\omega,\epsilon\in K_0$ with
$v_0(\omega)=-m$ and $v_0(\epsilon)>-u_2$.

We now show that without loss of generality we may assume
$v_0(\epsilon) \geq -2m$.  If $v_0(\epsilon) < -2m$, there are two
cases: If $v_0(\epsilon)$ is even, take $\eta \in K_0$ with $\eta^2
\omega^2 \equiv \epsilon \bmod \euP_0 \epsilon$. Then $-u_1/2 <
v_0(\eta)= v_0(\epsilon)/2 + m <0$, and $v_0(2 \eta) > v_0(2)-u_1/2
>0$. Set $x'_1=(1+2\eta)x_1 - \eta$. Then $v_1(x'_1)=v_1(x_1)=-u_1$,
and we calculate $\wp(x'_1)=\beta'$ where $\beta'=\beta +
\eta(\eta+1)(1+4\beta) \in K_0$.  We may therefore replace $x_1$ by
$x'_1$, $\beta$ by $\beta'=\wp(x'_1)$, and $\epsilon$ by
$\epsilon'=\wp(x_2)-\beta'\omega^2$ and find
$v_0(\epsilon')>v_0(\epsilon)$. If $v_0(\epsilon)$ is odd, take $\phi
\in K_0$ so that $\beta \phi^2 \equiv \epsilon \bmod \euP_0
\epsilon$. Then $v_0(\phi)= (v_0(\epsilon)+u_1)/2 > -m$ and also
$v_0(2 \beta \omega \phi) > v_0(\beta \phi^2)$ since $v_0(\epsilon) <
-2m \leq 0 < 2v_0(2)-u_2$. Then $\wp(x_2)=\beta(\omega+\phi)^2 +
\epsilon'$ with $\epsilon'=\epsilon - \beta \phi^2 -2\beta \omega
\phi$ where $v_0(\epsilon')>v_0(\epsilon)$.  We may therefore replace
$\omega$ by $\omega+\phi$ and $\epsilon$ by $\epsilon'$ and find
$v_0(\epsilon')>v_0(\epsilon)$.  Repeat these steps as necessary until
$v_0(\epsilon) \geq -2m$.

The existence of a Galois scaffold follows from Theorem \ref{elem-abel}
once we verify Assumptions \ref{elem-ab-1}, \ref{ass-epsilon} and
\ref{strong-ass-p}. Assumption \ref{elem-ab-1} is clear, and  
Assumption \ref{ass-epsilon} is the statement that $v_0(\epsilon) >
-u_2 + C_1 = -u_1/2 -2m$, which holds since $v_0(\epsilon) \geq -2m$. 
Assumption \ref{strong-ass-p} for $\euc\geq 1$ is equivalent to
$4v_0(2) \geq 2b_1+b_2+ \euc$. 
\end{proof}

\begin{proposition}\label{biquad-prop}
  Let $K_2$ be a totally ramified biquadratic extension of $K_0$, with
  lower ramification numbers satisfying $2b_1+b_2 < 4v_0(2)$.
\begin{itemize}
\item[(i)] If  $b_1 \equiv 1 \bmod 4$
  then $\euP_2^h$ is free over $\euA_h$ when $h \equiv 0,1 \bmod 4$
  and $2b_1+b_2 \leq 4v_0(2)-1$, or when $h \equiv 3 \bmod 4$ and
  $2b_1+b_2 \leq 4v_0(2)-5$. Moreover, $\euP_2^h$ is not free over
  $\euA_h$ when $h \equiv 2 \bmod 4$ and $2b_1+b_2 \leq 4v_0(2)-9$.
\item[(ii)] If $b_1 \equiv 3 \bmod
  4$ then $\euP_2^h$ is free
  over $\euA_h$ when $h \equiv 0, 2, 3 \bmod 4$ and $2b_1+b_2 \leq
  4v_0(2)-3$. Moreover, $\euP_2^h$ is not free
  over $\euA_h$ when $h \equiv 1 \bmod 4$ and $2b_1+b_2 \leq
  4v_0(2)-7$.
\end{itemize}
\end{proposition}

\begin{table}  \label{biq0}  
\centerline{ 
\begin{tabular}{|r|r| r r |} \hline
$b$ & $h$ & $L_1$  & $L_2$ \\ \hline
1 & 1  &  4 & 1 \\ \hline
1 & 0 & 5 & 1\\ \hline
1 & -1 & 6 & 2 \\ \hline
1 & -2 & 7 & 3 \\ \hline
3 & 3 & 4 & 1\\ \hline
3 & 2 & 5 & 1\\ \hline
3 & 1 & 6 & 2 \\ \hline
3 & 0  & 7 & 3  \\ \hline 
\end{tabular}
}  
\vskip5mm

\caption{Values of the bounds on $\euc$  in the
  biquadratic case.} 
\end{table}

\begin{proof}
The bounds on $2b_1+b_2$ ensure that we always have $b_1 \equiv b_2
\equiv 1,3 \bmod 4$. So let $b\in\{1,3\}$ with $b \equiv b_2
\bmod 4$. Freeness of $\euP_2^h$ over $\euA_h$ depends only on the
residue of $h\bmod 4$. So we may assume $0 \leq b-h <4$. This choice
of $b$ and $h$ is consistent with the notation of
\cite[\S3.1]{byott:A-scaffold}.  The 
values of $L_1:=4+b-h$ and $L_2:=\max(1,b-h)$ are as shown in Table
1.  A brute force check, as in the proof of \cite[Thm
  4.4]{byott:A-scaffold}, determines that the condition $w(s)=d(s)$ of
\cite[Thm 3.1]{byott:A-scaffold} holds for all $s$ except when
$b=1$, $h=-2$ or $b=3$, $h=1$. This gives the result; the
bounds on $2b_1 +b_2$ in each case come from the fact that $\euc$ in
Proposition \ref{biq-l-0} must be bounded below by either $L_1$ (for
the assertion that $\euP_2^h$ is not free) or $L_2$ (for the
assertion that it is free), together with the observation that
$2b_1+b_2 \equiv 3$ (respectively, 1) mod 4 if $b \equiv 1$
(respectively, 3) mod 4.
\end{proof}

\subsection{Weakly ramified extensions}

A Galois extension of local fields is said to be weakly ramified if
its second ramification group is trivial. An extension of global
fields is weakly ramified if all its completions are. In a weakly
ramified extension $L/K$ of odd degree, there is a fractional ideal of $\euO_L$
whose square is the inverse different. It was shown by B.~Erez
\cite{erez} that this ideal is locally free over the group ring
$\euO_K[\Gal(L/K)]$. This led several subsequent authors (see for
example \cite{vinatier}, \cite{pickett}) to investigate the square
root of the inverse different in weakly ramified extensions, both of
number fields and of local fields. The valuation ring, and its maximal
ideal, in a weakly ramified (but not necessarily totally ramified)
extension of local fields are studied as Galois modules in
\cite{johnston}.

Here we consider totally and weakly ramified Galois extensions $K_n/K_0$
of degree $p^n$, where $K_0$ is a local field whose residue field is
perfect of characteristic $p$. Thus $K_n/K_0$ is necessarily
elementary abelian. It is known that the valuation ring
$\euO_n$ is free over its associated order. This
can be proved using Lubin-Tate theory \cite[Cor 4.3]{byott:lubin-tate}
when the residue field is finite, but can also be deduced directly
from Erez' result; see also \cite{johnston}. In this section, we will
use \cite[Thm 3.1]{byott:A-scaffold} to give an alternative proof of
this result, while at the same time determining the structure of the
other ideals. Thus we define
$$\euA_h=\{\alpha\in K_0[G]:\alpha\euP_n^h\subseteq\euP_n^h\}.$$
Again, note that $\euA_0=\euA_{L/K}$.
We begin with the
fact that a Galois scaffold exists. Note that in 
characteristic $p$, $K_n/K_0$ has a Galois scaffold of
precision $\infty$ \cite{elder:scaffold}.  

\begin{proposition}\label{weakly0}
Let $K_0$ be a local field of characteristic $0$ whose residue field is
perfect of characteristic $p$, and let $K_n/K_0$ be a totally and
weakly ramified Galois extension of $K_0$ of degree $p^n$. Then $K_n/K_0$ has a Galois
scaffold of precision $\euc = p^n v_0(p)-(p^n-1)$.
\end{proposition} 
\begin{proof}
The hypothesis means that $K_n/K_0$ is elementary abelian of degree
$p^n$, with $b_i=u_i=1$ for $1 \leq i \leq n$. When $n=1$, the result follows from
Theorem \ref{C_p}. When  $n
\geq 2$, the result will follow from
Theorem \ref{elem-abel} provided that Assumptions \ref{elem-ab-1},
\ref{ass-epsilon} and \ref{strong-ass-p} hold. Using \cite[III \S2.5
  Prop]{fesenko},
$K_n=K_0(x_1,\ldots ,x_n)$ where
$\wp(x_i)=\omega_i^{p^{n-1}}\beta+\epsilon_i$ for some $\beta\in K_0$
with $v_0(\beta)=-1$, some $\omega_i\in\euO_0\setminus\euP_0$ which, in the
residue field $\euO_0/\euP_0$, are linearly independent over
$\mathbb{F}_p$, and some $\epsilon_i\in\euO_0$. Since $b_i=u_i=1$ for all
$i$, 
Assumptions \ref{elem-ab-1} and
\ref{ass-epsilon} hold. Furthermore Assumption 6 holds
since by choice of $\euc$, $v_0(p)= 1-1/p^n+\euc/p^n=C_n+\euc/p^n$. 
\end{proof}

We now give a new proof of the fact that the valuation ring of a
totally and weakly ramified extension is free over its associated
order, and that the square root of the inverse different is free
over the group ring. Our proof depends on the existence of a Galois
scaffold, and works both in characteristic $p$ and characteristic $0$,
with no hypothesis on $v_0(p)$. 

\begin{proposition} \label{weakly-struct-both}
Let $K_n/K_0$ be a totally and weakly ramified Galois extension of
local fields, and let $G=\Gal(K_n/K_0)$. Then $\euP_n$ (respectively,
$\euO_n$) is free over $\euO_0[G]$ (respectively, over $\euO_0[G]
[\pi_0^{-1} \mbox{\rm Tr}_{n,0}]$, where $\mbox{\rm Tr}_{n,0}=\sum_{g \in G} g$ is the trace
element in $K_0[G]$), and any element of $K_n$ of valuation $1$ is a
generator.
\end{proposition}
\begin{proof}
We are interested in the Galois module structure of $\euP_n^h$ for
$h=0$ and $h=1$. Since $b_i=1$ for all $i$, we may assume that
$b=1$. Our notation, namely $h,b$, is consistent with the notation in
\cite[\S3.1]{byott:A-scaffold}. Using Proposition \ref{weakly0},
$K_n/K_0$ has a Galois scaffold of precision $\euc \geq
1=\max(b-h,1)$. The numbers $d(s)$ and $w(s)$ occurring in  
\cite[Thm 3.1]{byott:A-scaffold} are easy to determine in this case
(as in the proof of \cite[Thm 4.5]{byott:A-scaffold}); we find when
$h=1$ that $d(s)=w(s)=0$ for all
$s\in\mathbb{S}_{p^n}=\{s:0\leq s<p^n\}$, and when $h=0$ that 
$d(s)=w(s)=0$ for all $s \neq p^n-1$ and $d(p^n-1)=w(p^n-1)=1$. Since
$d(s)=w(s)$ for all $s$, both $\euP_n$ and $\euO_n$ are free over
their associated orders by \cite[Thm 3.1(iv)]{byott:A-scaffold}. We
now identify those associated orders.
For $h=1$, we have $w(s)=0$ for all $s$, so $\euA_1 =
\euO_0[\Psi_1, \ldots, \Psi_n]$. Since in the proof of Proposition
\ref{weakly0} and in \cite[Prop 5.3]{elder:scaffold}, we have $\omega_i \in \euO_0$ for all $i$, it follows
from the construction of the Galois scaffold that $\Psi_i \in
\euO_0[G]$ for all $i$. Hence $\euA_1 \subseteq \euO_0[G]$. Since the
reverse inclusion certainly holds, we have $\euA_1=\euO_0[G]$ when
$h=1$. When $h=0$, we have $w(s)=0$ for $s \neq p^n-1$ and
$w(p^n-1)=1$, so $\euA_0=\euO_0[G] + \euO_0 \pi_0^{-1}
\Psi^{(p^n-1)}$. Thus $\euA_0/\euO_0[G]$ has dimension 1 over
$\euO_0/\euP_0$. But $\pi_0^{-1} \mbox{Tr}_{n,0} \in \euA_0$ since
$\mbox{Tr}_{K_n/K_0}(\euO_n) \subsetneq \euO_0$ because $K_n/K_0$ is wildly
ramified.  It follows that $\euA_0=\euO_0[G] + \euO_0 \pi_0^{-1} \mbox{Tr}_{n,0}$.

We next show that any $\pi_n \in K_n$ of valuation $1$ is a free
generator for $\euP_n$ and $\euO_n$.  This follows from \cite[Thm
  3.1(ii)]{byott:A-scaffold} if $v_0(p)\geq 2$, because then $\euc\geq
p^n+b-h$. So we need only consider the case $v_0(p)=1$, when $K_0$ has
characteristic $0$ and is unramified over the $p$-adic numbers. We nevertheless
give a proof which works more generally.
We may write $\pi_n=
\sum_{i=1}^{p^n}a_i\lambda_{i}$ for some $a_i\in\euO_0$
with $v_0(a_{0})=0$. Since, by Theorem \ref{main},
the $\Psi_j$ are $K_0$-linear maps, we observe as in \cite[(5)]{byott:A-scaffold} that 
$v_n(\Psi^{(j)}\lambda_i)\geq i+j$ with $v_n(\Psi^{(j)}\lambda_1)= 1+j$. 
It follows that 
$v_n(\Psi^{(s)} \pi_n) =
s+1$ for all $s \in \Spn$. Thus $\euO_0[G] \pi_n = \euP_n$
and $(\euO_0[G] + \euO_0 \pi_0^{-1}\mbox{\rm Tr}_{n,0}) \pi_n = \euO_n$, as required.
\end{proof}

\begin{remark}
The valuation of the different of $K_n/K_0$ is $2(p^n-1)$, so the
square root of the inverse different is $\euP_n^{1-p^n}$, which is
isomorphic to $\euP_n$. The fact that the square root of the inverse
different is free over the group ring $\euO_0[G]$ therefore follows
from the case $h=1$ in
Proposition \ref{weakly-struct-both}.
\end{remark}

We now use the Galois scaffold of Proposition \ref{weakly0} to
determine which other ideals $\euP_n^h$ are free over their associated
orders.  

\begin{proposition} \label{weakly-struct0}
Let $K_n/K_0$ be as in Proposition \ref{weakly0} with $v_0(p) \geq 3$.
Then $\euP_L^h$ is free over 
its associated order if and only if $h \equiv h' \bmod p^n$ where
$h'=0$, $h'=1$, or $\frac{1}{2}(p^n+1) < h' < p^n$.
\end{proposition}
\begin{proof}
The condition $v_0(p)\geq 3$ ensures that 
$K_n/K_0$ has a Galois scaffold of precision $\euc
\geq 2p^n-1$. We can therefore apply \cite[Thm 3.1(ii)]{byott:A-scaffold}. The condition on $h$ then follows as in
\cite[Thm 4.5]{byott:A-scaffold}. 
\end{proof}

\subsection{Abrashkin's elementary extensions}
Let $K_0$ be a local field of characteristic $0$ with perfect residue
field of characteristic $p$ containing the field of $p^n$
elements. Following \cite[III\S2 Ex.4]{fesenko}, we define $K_n$ to be
an elementary extension of $K_0$ if $K_n=K_0(x)$ where
$x^{p^n}-x=\tau$ with $\tau \in K_0$ and $v_0(\tau) >
-p^nv_0(p)/(p^n-1)$. 

In the case that $K_0$ is unramified over
$\bQ_p$, the elementary extensions were introduced by V.~Abrashkin
\cite{abrashkin}, who used them in his proof that there are no abelian
schemes over $\bZ$.

We set $u=-v_0(\tau)$, and make the further assumptions that $u>0$,
$\gcd(u,p)=1$. Then $K_n/K_0$ is a totally ramified elementary abelian
extension of degree $p^n$, with unique ramification break $u$.
Furthermore $K_n/K_0$ has a Galois scaffold, directly generalizing
Theorem \ref{C_p}, which concerns the case $n=1$.

\begin{proposition}
If $K_n/K_0$ is an elementary abelian extension as above, then
$K_n/K_0$ has a Galois scaffold of precision $\euc = p^n v_0(p) -
(p^n-1)u \geq 1$.
\end{proposition}
\begin{proof}
Let $\mathbb{F}_p$ be the finite field with $p$ elements and
$\mathbb{F}_{p^n}$ the finite field with $p^n$ elements contained in
the residue field $\euO_0/\euP_0$.  Let $\omega_1=1, \omega_2, \ldots ,
\omega_n\in\euO_0$ be Teichm\"{u}ller representatives for an
$\mathbb{F}_p$-basis of $\mathbb{F}_{p^n}$.  So
$\omega_i^{p^n}=\omega_i$.  We prove that $K_n=K_0(x_1,\ldots, x_n)$
where $x_i^p-x_i=\omega_i\tau$.
But this follows if we prove that for each $i$, $z^p-z= \omega_i\tau$ has
$p$ solutions in $K_n$. Consider the polynomial
$$f(y)=\left (\sum_{r=0}^{n-1}(\omega_ix)^{p^r}+y\right )^p-\left
(\sum_{r=0}^{n-1}(\omega_ix)^{p^r}+y\right )-\omega_i\tau\in K_n[y].$$
The bound on $t$ means that $f(y)\equiv y^p-y\bmod \euP_n$, so
$f(y)$ has $p$ roots in $K_n$ by Hensel's Lemma.

We have ramification numbers $b_i=u_i=u$ for all $i$. As
$\gcd(u,p)=1$, Assumption \ref{elem-ab-1} is satisfied, and Assumption
\ref{ass-epsilon} is satisfied since $\epsilon_i=0$ for all
$i$. Finally, Assumption 6 with precision $\euc$ is equivalent to
$p^nv_0(p) \geq (p^n-1)u + \euc$. This holds for $\euc = p^n v_0(p) -
(p^n-1)u$, and then $\euc \geq 1$ by the condition on $v_0(\tau)$. 
\end{proof}

\begin{corollary} \label{abrashkin-cor} 
Let $K_n/K_0$ be an elementary extension as above with $u>0$,
$\gcd(u,p)=1$, and suppose that $u$ satisfies the slightly more
restrictive condition
$$ u < \frac{p^n v_0(p)}{p^n-1} - 2. $$
Then the freeness or otherwise of any fractional ideal $\euP_n^h$ of
$\euO_n$ is determined by the numerical data $d(s)$, $w(s)$ as in 
\cite[Thm 3.1]{byott:A-scaffold}.
In particular, $\euO_n$ itself is free over its associated
order $\euA_0=\euA_{L/K}$ if the least non-negative residue $b$ of $u$ mod $p^n$ divides
$p^m-1$ for some $m \leq n$. For $n=2$, $\euO_n$ is free over $\euA_0$
if and only if $b$ divides $p^2-1$.   
\end{corollary}
\begin{proof}
The condition on $u$ ensures that there is a Galois scaffold of
precision $\euc \geq 2p^n-1$, so that all parts of \cite[Thm
  3.1]{byott:A-scaffold} apply. The statements about $\euO_n$ follow
from \cite[Cor 3.3]{byott:A-scaffold}.
\end{proof}

\begin{remark}
The corresponding extensions in characteristic $p$ have a Galois
scaffold of precision $\infty$ \cite[Lemma 5.2]{elder:scaffold}. Hence
the conclusions of Corollary \ref{abrashkin-cor} hold for these
extensions as well.
\end{remark}

\section{New directions: Constructing Hopf orders in $K[G]$}\label{Hopf-orders} 
The purpose of this section is to illustrate how the results of this
paper can be used to construct commutative and cocommutative Hopf
algebras over valuation rings of local fields and thus shed light on
the construction of finite abelian group schemes over such rings.  Our
purpose here is not to be exhaustive, but simply to
illustrate the utility of a Galois scaffold outside of the Galois
module structure of ideals in local field extensions.

\subsection{Background}
Let $K$ be a local field of residue characteristic $p>0$, and let $G$
be a finite abelian $p$-group. There are not many results that
classify Hopf algebras defined over $\euO_K$ within $K[G]$ ({\em
  i.e.}~Hopf orders). If $G$ has order $p$, the classification of Hopf
orders follows from work of Tate and Oort on group schemes of rank $p$
\cite{tate:oort}. If $\chr(K)=0$ and $K$ contains the $p$th roots of
unity, Greither classified a family of Hopf orders for $G\cong
C_{p^2}$ \cite{greither:MZ}. Under the same assumptions, the
first-named author classified all Hopf orders for $G\cong C_p\times
C_p$, and assuming that $K$ contains the $p^2$th roots of unity, all
Hopf orders for $G\cong C_{p^2}$ \cite{byott:cleft1,
  byott:cleft2}. More recently, Tossici classified all Hopf orders for
$G$ of order $p^2$ without any restriction on the field $K$
\cite{tossici:p^2}. Motivated by the construction in this section,
specifically Example \ref{2}, the classification for $G\cong C_p\times
C_p$ and $\chr(K)=p$ was reproven using Greither's approach
\cite{elder:underwood}.  For $G$ of order $p^n$, $n\geq 3$ there are
families of Hopf orders for $G$ cyclic
\cite{underwood:96,childs:under03,childs:under04,underwood:childs} and
for $G$ elementary abelian
\cite{childs:sauerberg,greither:childs,childs:elem-abel}. For $n=3$,
these families are known to be incomplete. More recently, M\'{e}zard,
Romagny and Tossici produce a family of Hopf orders for $G$ cyclic
that they conjecture is complete \cite{tossici:p^n}. Still for $p=3$,
the conjecture is unproven.  Additional methods are needed.

\subsection{A method for constructing Hopf orders}\label{method}
Let $K$ be a local field with residue characteristic $p>0$. In
particular, $K$ can have characteristic $0$ or $p$.  Let $L/K$ be a
totally ramified extension of degree $p^n$, $n>1$ with abelian Galois
group $G$, and let the lower ramification numbers $b_i$ satisfy
$b_i\equiv -1\bmod p^n$ (thus Assumptions \ref{ass1} and \ref{ass2}
hold with $b=p^n-1$). Assume now that it is possible to make Choices
\ref{choice1} and \ref{choice2} so that Assumption \ref{ass3} holds
with $\euc\geq p^n-1$. Under these circumstances, Theorem \ref{main}
states that a scaffold exists for the action of $K[G]$ on $L/K$. As
explained in \cite[Remark 3.5]{byott:A-scaffold}, it then follows from  
\cite[Theorem 3.1]{byott:A-scaffold} that the ring of integers
$\euO_L$ is free over its associated order $\euA_{L/K}$, and also that 
this associated order takes a particularly simple form:
\begin{equation}\label{hopf-order}
\euA_{L/K}=\euO_K\left[\frac{\Theta_n-1}{\pi_K^{M_n}}, \ldots
  ,\frac{\Theta_2-1}{\pi_K^{M_2}}, \frac{\Theta_1-1}{\pi_K^{M_1}} \right],
\end{equation}
for integers $M_i\geq 0$ satisfying $M_i=(b_i+1)/p^i$. Furthermore,
the $\Theta_i$ are as defined in Definition \ref{Theta}, with
\begin{equation}\label{mu-equal}
v_K(\mu_{i,j})=(b_i-b_j)/p^j=p^{i-j}M_i-M_j.
\end{equation}

Now assume the very weak condition that largest lower ramification
number $b_n$ satisfies $b_n-\lfloor b_n/p\rfloor\leq p^{n-1}v_K(p)$,
where $\lfloor\cdot\rfloor$ denotes the greatest integer
function. This condition is empty in characteristic $p$, and because of
the congruences on $b_i$, it eliminates only certain cyclic extensions in
characteristic $0$ \cite[Proposition 3.7]{byott:JNTB}. Under this
assumption, $\euA_{L/K}$ is a local ring (equivalently, $\euO_L$ is
indecomposable as an $\euO_K[G]$-module) \cite[Theorem 3]{vostokov:1},
\cite[Theorem 4]{vostokov:2}. The congruence conditions on the $b_i$
can then be used together with \cite[IV \S2 Proposition
  4]{serre:local} to see that the different satisfies
$\euD_{L/K}=\delta\euO_L$ for some $\delta\in K$. We are now in a position to use a result of Bondarko, and so, since
$\euO_L$ is free over $\euA_{L/K}$, the
associated order $\euA_{L/K}$ is a Hopf order \cite[Thm A, Prop
  3.4.1]{bondarko}.

\begin{remark}
A Hopf order $\euH$ in a group algebra $K[G]$ is said to be
realizable if there is an extension $L/K$ such that $\euO_L$ is an
$\euH$-Hopf Galois extension of $\euO_K$. When $\euH$ is itself a
commutative local ring, this is equivalent to saying that $\euH$ arises as the
associated order of some valuation ring $\euO_L$. Moreover, in this
case, $\euH$ is realizable if and only if its dual Hopf order is
monogenic as an $\euO_K$-algebra \cite{byott:monogenic}. The Hopf
orders $\euA_{L/K}$ we consider here are, by construction, associated
orders of valuation rings, and hence are realizable.  
\end{remark}

\subsection{Hopf orders in $K[C_p^n]$}
In the case that $K$ has characteristic $0$ and $G$ is elementary
abelian, several families of Hopf orders have been described
\cite{childs:sauerberg,greither:childs,childs:elem-abel}.  It is
therefore of interest to describe the elementary abelian Hopf orders
that result from this paper in a little more detail. But, although we
provide some superficial comments regarding our Hopf orders and those
in \cite{childs:elem-abel}, we leave a careful comparison of the
families of Hopf orders for a later paper.

We first construct some elementary abelian extensions. Again, $K$ may be of either characteristic.  Choose an
integer $m_1\geq 1$ and integers $m_i\geq 0$ for $i\geq 2$ so that
\begin{equation}\label{Hopf-vp}
v_K(p)\geq m_1(p^n-1) +\sum_{k=2}^n(p^{n-1}-p^{k-2})m_k.
\end{equation}
No such integers are possible unless $v_K(p)$ is big enough, namely $v_K(p)\geq p^n-1$.  The restriction is vacuous if $K$
has characteristic $p$.  Note that if we can now arrange for our
extension to have upper and lower ramification numbers as in
\S\ref{elem-ab-main} with $b_1=p^nm_1-1$, then \eqref{Hopf-vp} is
equivalent to $v_K(p)\geq u_n-b_n/p^n+\euc/p^n$, Assumption
\ref{strong-ass-p}, with $\euc=p^n-1$.

Choose $\beta\in K$ with $v_K(\beta)=-b_1=1-m_1p^n$. Thus Assumption \ref{elem-ab-1} holds. Choose
elements $\omega_i\in K$ for $i\geq 2$ such that
$v_K(\omega_i)=-\sum_{k=2}^im_i$, and furthermore assume that, whenever
$v_K(\omega_i)=\cdots=v_K(\omega_j)$ with $i<j$, the
projections of $\omega_i, \ldots ,\omega_j$ into
$\omega_i\euO_K/\omega_i\euP_K$ are linearly independent over
$\mathbb{F}_p$, the finite field with $p$ elements.  This last
assumption requires $[\euO_K/\euP_K:\mathbb{F}_p]$ to be strictly
larger than the longest string of consecutive zeros in $(m_2,\ldots,
m_n)$.  Let $L=K(x_1,\ldots,x_n)$ where $\wp(x_1)=\beta$ and
$\wp(x_i)=\omega_i^{p^{n-1}}\beta$ for $2\leq i\leq n$. All this is in
agreement with \S\ref{elem-ab-main}, except that because we have
chosen $\epsilon_i=0$, Assumption \ref{ass-epsilon} is vacuous.

At this point, we have $L/K$, a totally ramified elementary abelian
extension of degree $p^n$ with $\mbox{Gal}(L/K)=\langle\sigma_1,\ldots
,\sigma_n\rangle$ where $(\sigma_i-1)x_j\equiv \delta_{ij}\bmod
\euP_L$. This extension has a Galois scaffold of precision
$\euc=p^n-1$ in characteristic $0$ and $\euc=\infty$ in
characteristic $p$.   And based upon
\S\ref{method}, we have a Hopf order of the form
\eqref{hopf-order}. 

An explicit description of this Hopf order however,
requires that we describe the integers $M_i$, as well as the elements
$\mu_{i,j}\in K$ that are used to define the $\Theta_i$ in Definition
\ref{Theta}.  Since
\begin{equation}\label{M-m}
p^iM_i=b_i+1=p^n(m_1+\sum_{k=2}^ip^{k-2}m_k),
\end{equation} these nonnegative
integers satisfy
\begin{equation}\label{M-inequ}
p^rM_r\leq p^sM_s,
\end{equation}
 for all $1\leq r\leq s\leq n$.  (Note that
$i_j=M_{n-j+1}$ translates our notation into analogous notation, namely {\em valuation
  parameters}, in \cite[p. 491]{childs:elem-abel}.
In their terms, our condition becomes
$p^{s-r}i_r\geq i_s$ for $r\leq s$, which is weaker that their
requirement, $pi_r\geq i_s$ for $r< s$.)  Meanwhile, \eqref{Hopf-vp}, which is vacuous in characteristic $p$, can be rewritten, using
\eqref{M-m}:
\begin{equation}\label{M-bound-p}
\frac{v_K(p)}{p-1}\geq \sum_{i=1}^nM_i.
\end{equation}
(This is stronger than the condition $v_K(p)/(p-1)>i_j$ in
\cite{childs:elem-abel}.) We note that there are also congruence
conditions imposed by \eqref{M-m}, which we ignore, recalling Remark
\ref{improvements}, until we can say which are an artifact of
our scaffold construction, and which are not.

Turning to the $\mu_{i,j}$, superficial comparisons with analogous
expressions in \cite{childs:elem-abel} become impossible.
A detailed translation would greatly expand the scope of the paper.
So
we close with two examples, $n=2,3$, where we include all the details.
\begin{example}\label{2} Let $C_p^2=\langle\sigma_2,\sigma_1\rangle$. Choose any integers 
$M_i\geq 0$ such that $M_1\leq pM_2$ and $p\mid M_1$, and let
  $\mu_{1,2}\in K$ be any element with
  $v_K(\mu_{1,2})=M_1/p-M_2$.  Then
$$\euO_K\left[\frac{\sigma_2-1}{\pi_K^{M_2}},\frac{\sigma_1\sigma_2^{[-\mu_{1,2}]}-1}{\pi_K^{M_1}}\right]$$
  is a Hopf order in $K[C_p^2]$ when $K$ has characteristic $p$. It is
  a Hopf order in $K[C_p^2]$ when $K$ has characteristic $0$ under the
  additional assumption that $v_K(p)/(p-1)\geq M_1+M_2$. Note that the
  valuation of $\mu_{1,2}$ makes $p\mid M_1$ redundant.
\end{example}

\begin{example}\label{3}
Let $C_p^3=\langle\sigma_3,\sigma_2,\sigma_1\rangle$. Choose any
integers $M_i\geq 0$ such that $pM_1\leq p^2M_2\leq p^3M_3$ and
$p^{3-i}\mid M_i$, but $p\mid M_3$.  Let $\mu_{i,j}\in K$ be elements
with $v_K(\mu_{i,j})=p^{i-j}M_i-M_j$ that arise from elements
$\omega_1,\omega_2\in K$ via the matrix $(\mathbf{\Omega})$ defined in
\eqref{matrix-eq}, as explained in the proof of Lemma \ref{assum3}. In
other words,
\begin{equation}\label{intertwining}
\mu_{1,2}=-\omega_2^p,\quad\mu_{2,3}=-\frac{\omega_3^p-\omega_3}{\omega_2^p-\omega_2},\quad\mu_{1,3}=-\frac{\omega_3\omega_2^p-\omega_2\omega_3^p}{\omega_2^p-\omega_2}.
\end{equation}
Then
$$\euO_K\left[\frac{\sigma_3-1}{\pi_K^{M_3}},
  \frac{\sigma_2\sigma_3^{[-\mu_{2,3}]}-1}{\pi_K^{M_2}},
  \frac{\sigma_1\sigma_3^{[-\mu_{1,3}]}(\sigma_2\sigma_3^{[-\mu_{2,3}]})^{[-\mu_{1,2}]}-1}{\pi_K^{M_1}}\right]$$
is a Hopf order in $K[C_p^3]$ when $K$ has characteristic $p$. It is a
Hopf order in $K[C_p^3]$ when $K$ has characteristic $0$ under the
additional assumption that $v_K(p)/(p-1)\geq M_1+M_2+M_3$.  Note that the
valuation of $\mu_{i,j}$ makes $p^{3-i}\mid M_i$, but not $p\mid M_3$,
redundant. Thus we suspect $p\mid M_3$ to be an artifact of our
scaffold construction. We are uncertain as to the the significance of
\eqref{intertwining}.
\end{example}
\begin{remark}
Although we are not yet prepared to carefully analyze the Hopf orders
described in this section in comparison with those that appear in
\cite{childs:sauerberg,greither:childs, childs:elem-abel}, we can say
that the Hopf orders in Example \ref{3} do not appear in
\cite{childs:elem-abel}. The Hopf orders in Example \ref{3} are
realizable, while the Hopf orders with $n=3$ in \cite{childs:elem-abel} are
not. See \cite[Proposition 15]{childs:elem-abel}.
\end{remark}

\bibliography{bib}

\end{document}